\documentclass[noinfoline,aop]{imsart}
\usepackage{amssymb, fixmath}
\usepackage{amsmath}
\usepackage{amsthm}
\usepackage{epsfig}  		
\usepackage{graphicx,float, dsfont}
\usepackage{hyperref}
\usepackage{enumitem}

\RequirePackage[OT1]{fontenc}
\RequirePackage{amsthm,amsmath}
\RequirePackage[numbers]{natbib}
\RequirePackage[colorlinks,citecolor=blue,urlcolor=blue]{hyperref}

\usepackage{pslatex}
 \usepackage{lineno}
 \usepackage{color}

\newcommand{\E}{\mbox{\bf E}}
\def\P{{\bf P}}

\def\ind{{\mathbf 1}}

\def\to{\rightarrow}

\def\d{\partial}

\def\bar{\overline}

\def\l{\left}

\def\<{\langle}
\def\>{\rangle}

\newcommand\mnote[1]{} 
\newcommand{\beq}[1]{\begin{equation}\label{#1}}
\newcommand\eeq{\end{equation}}
\newcommand\ben{\begin{equation}}
\newcommand\een{\end{equation}}
\newcommand\bes{\begin{eqnarray*}}
\newcommand\ees{\end{eqnarray*}}
\newcommand\besn{\begin{eqnarray}}
\newcommand\eesn{\end{eqnarray}}
\newcommand\beal{\begin{align*}}
\newcommand\eaal{\end{align*}}

\def\bthm{\begin{theorem}}
\def\ethm{\end{theorem}}
\def\bdefn{\begin{definition}}
\def\edefn{\end{definition}}
\def\benu{\begin{enumerate}}
\def\eenu{\end{enumerate}}
\def\beit{\begin{itemize}}
\def\eeit{\end{itemize}}
\def\beds{\begin{description}}
\def\eeds{\end{description}}
\def\bepr{\begin{problem}}
\def\eepr{\end{problem}}

\def\bprf{\begin{proof}}
\def\eprf{\end{proof}}
\def\berk{\begin{remark}}
\def\eerk{\end{remark}}
\def\bex{\begin{exercise}}
\def\eex{\end{exercise}}
\def\beg{\begin{example}}
\def\eeg{\end{example}}

\renewcommand{\qed}{\hfill\text{$\blacksquare$}}

\def \CC{{\mathcal C}}
\def \DD{{\mathcal D}}

\def\BB{{\mathcal B}}
\def\AA{{\mathcal A}}

\def\LLL{{\mathcal L}}

\def\PP{{\mathcal P}}
\def\NN{{\mathcal N}}

\def\N{\mathbb{N}}
\def\R{\mathbb{R}}

\def\Z{\mathbb{Z}}

\newcommand{\sm}{{\raise0.3ex\hbox{$\scriptstyle \setminus$}}}

\renewcommand{\phi}{\varphi}

\def\CHI{\mathchoice%
{\raise2pt\hbox{$\chi$}}%
{\raise2pt\hbox{$\chi$}}%
{\raise1.3pt\hbox{$\scriptstyle\chi$}}%
{\raise0.8pt\hbox{$\scriptscriptstyle\chi$}}}
\def\smalloplus{\raise1pt\hbox{$\,\scriptstyle \oplus\;$}}

\usepackage{mathtools}
\usepackage{pslatex}

\def \sas {S$\alpha$S }

\def \1 {\mathbf{1}}
\def \l {\ell}
\DeclarePairedDelimiter\floor{\lfloor}{\rfloor}

\usepackage{dsfont}
\usepackage[utf8]{inputenc}

\usepackage{amsmath,amsthm,amssymb,color,tikz,epsfig,amsfonts,latexsym,  bbm, mathrsfs}
\usepackage{epsfig,amsfonts,latexsym}

\usepackage{graphicx}
\usepackage{xcolor}

\numberwithin{equation}{section}

\usetikzlibrary{matrix}

\usepackage{bbm}

\numberwithin{equation}{section}

\allowdisplaybreaks

\newcommand{\been}{\begin{enumerate}}

\usepackage{cleveref}
\newtheorem{thm}{Theorem}[section]

\newtheorem{lemma}[thm]{Lemma}

\newtheorem{cor}[thm]{Corollary}

\theoremstyle{remark}
\newtheorem{remark}[thm]{Remark}

\newtheorem{problem}[thm]{Problem}
\newtheorem{example}{Example}[section]

\theoremstyle{definition}

\usepackage{mathtools}
\usepackage{pslatex}

\def \sas {S$\alpha$S }

\def \ind {\mathbf{1}}
\def \l {\ell}
\def \Z{\mathbb{Z}}
\def \N{\mathbb{N}}
\def \E{\mathbb{E}}

\begin{document}

\begin{frontmatter}
\title{Stable Random Fields Indexed by Finitely Generated Free Groups
}
\runtitle{Stable Random Fields on Free Groups}

\begin{aug}
\author{\fnms{Sourav} \snm{Sarkar}\thanksref{t1,m1}\ead[label=e1]{souravs@berkeley.edu}
}
\and
\author{\fnms{Parthanil} \snm{Roy}\thanksref{t2,m2}
\ead[label=e2]{parthanil.roy@gmail.com}
}

\thankstext{t1}{Sourav Sarkar (corresponding author) was supported in part by Lo\`{e}ve Fellowship at University of California, Berkeley.}
\thankstext{t2}{Parthanil Roy was supported by Cumulative Professional Development Allowance from Government of India and the project RARE-318984 (a Marie Curie FP7 IRSES Fellowship).}
\runauthor{S. Sarkar and P. Roy}

\affiliation{University of California, Berkeley\thanksmark{m1} and Indian Statistical Institute, Bangalore \thanksmark{m2}}

\address{Sourav Sarkar\\
Graduate Student\\
Department of Statistics\\
University of California, Berkeley\\
337 Evans Hall\\
Berkeley, CA 94720-3860\\
USA.\\
\printead{e1}
}

\address{Parthanil Roy\\
Associate Professor\\
Theoretical Statistics and Mathematics Unit\\
Indian Statistical Institute, Bangalore Centre\\
8th Mile, Mysore Road\\
Bangalore 560059\\
India.\\
\printead{e2}
}
\end{aug}

\begin{abstract}
In this work, we investigate the extremal behaviour of left-stationary \emph{symmetric $\alpha$-stable} (S$\alpha$S) random fields indexed by finitely generated free groups. We begin by studying the rate of growth of a sequence of partial maxima obtained by varying the indexing parameter of the field over balls of increasing size. This leads to a phase-transition that depends on the ergodic properties of the underlying nonsingular action of the free group but is different from what happens in the case of S$\alpha$S random fields indexed by $\mathbb{Z}^d$. The presence of this new dichotomy is confirmed by the study of a stable random field induced by the canonical action of the free group on its Furstenberg-Poisson boundary with the measure being Patterson-Sullivan. This field is generated by a conservative action but its maxima sequence grows as fast as the i.i.d.~case contrary to what happens in the case of ${\Z}^d$. When the action of the free group is dissipative, we also establish that the scaled extremal point process sequence converges weakly to a novel class of point processes that we have termed as \emph{randomly thinned cluster Poisson processes}. This limit too is very different from that in the case of a lattice.

\end{abstract}



\end{frontmatter}




\section{Introduction} \label{sec:intro}
A random variable $X$ is said to follow symmetric $\alpha$-stable (S$\alpha$S) distribution
($\alpha \in (0, 2]$, the index of stability) with scale parameter
$\sigma > 0$ if it has characteristic function of the form
$E(e^{i\theta X}) = \exp{\{-\sigma^{\alpha}|\theta|^\alpha\}}$, $\theta \in \mathbb{R}$.  In this work, we will always concentrate on the non-Gaussian case, i.e., $\alpha \in (0,2)$. For encyclopedic treatment of $\alpha$-stable ($0<\alpha <2$) distributions and processes, we refer the readers to \cite{samorodnitsky:taqqu:1994}. A random field $\mathbf{X}=\{X_t\}_{t\in G}$, indexed by a (possibly noncommutative) countable group $(G,\cdot)$ is called an \sas random field if for each $k \geq 1$, for each $t_1, t_2, \ldots, t_k \in G$ and for each $c_1, c_2, \ldots, c_k \in \mathbb{R}$, the linear combination $\sum_{i=1}^k c_i X_{t_i}$  follows an \sas distribution. Also $\{X_t\}_{t \in G}$ is called left-stationary, if $\{X_t\} \overset{d}{=} \{X_{s \cdot t}\} $ for all $s\in G$. The notion of right-stationarity can be defined analogously and will coincide with left-stationarity when $G$ is abelian. Whatever we prove for left-stationary \sas random fields will have their corresponding counterparts in the right-stationary case. From now on, we shall write stationary to mean left-stationary throughout this paper.

Thanks to the seminal works of Rosi{\'n}ski \cite{Rosinski:1994}, \cite{Rosinski:1995}, \cite{Rosinski:2000}, various probabilistic aspects of stationary \sas random fields indexed by $\mathbb{Z}$ or $\mathbb{Z}^d$ have been connected to the ergodic theoretic properties of the underlying nonsingular group action; see, for example, \cite{rosinski:samorodnitsky:1996}, \cite{mikosch:samorodnitsky:2000a}, \cite{samorodnitsky:2004a}, \cite{Resnick:Samorodnitsky:2004},   \cite{samorodnitsky:2005a}, \cite{cohen:samorodnitsky:2006}, \cite{roy:samorodnitsky:2008}, \cite{Roy:2010}, \cite{Wang:Yizao:Roy:2013}, \cite{owada:samorodnitsky:2015a}, \cite{fasen:roy:2016}. For similar connections in case of max-stable processes and fields, we refer the readers to \cite{stoev:taqqu:2005}, \cite{stoev:2008}, \cite{kabluchko:2009}, \cite{kabluchko:schlather:dehaan:2009}, \cite{kabluchko:schlather:2010}, \cite{wang:stoev:2010a}, \cite{wang:stoev:2010b}, \cite{dombry:kabluchko:2014}, \cite{dombry:kabluchko:2016}. See also \cite{roy:2007a}, \cite{owada:samorodnitsky:2015b}, \cite{jung:owada:samorodnitsky:2015}, \cite{kabluchko:stoev:2016} for links between ergodic theory and stationary infinitely divisible processes, and \cite{roy:2012} for an alternative approach to stable processes using Maharam systems.

In all the works mentioned above, the indexing group $G$ is $\mathbb{Z}^d$ (or $\mathbb{R}^d$ in the continuous parameter case) for some $d \geq 1$ and hence amenable. Many of the proofs use the amenability of the underlying group in some way or the other. In the present work, we would like to go beyond the framework of amenable groups and study the corresponding stable random fields. To this end, we first establish a general phase transition result (see Theorem~\ref{decomp} below) for extremes of stable fields indexed by finitely generated countable groups, and then concentrate on the simplest possible class of non-amenable groups, namely, the finitely generated free groups. We use nonsingular actions of free groups to construct stationary S$\alpha$S random fields in parallel to \cite{Rosinski:1995, Rosinski:2000} and investigate the extremal properties of such fields in details under various ergodic theoretic conditions on the action.

The motivation behind our work is twofold. Firstly, ergodic theoretic properties of group actions may change significantly as we pass from amenable to non-amenable groups; (see, for instance, \cite{Tao:2015} for a recent article which shows that the pointwise and maximal ergodic theorems do not hold in $L^1$ for measure-preserving actions of finitely generated free groups). This necessitates the investigation of the effect of the ergodic theoretic change on various probabilistic aspects of the stable fields and finitely generated free groups serve as a convenient test-case in the class of non-amenable groups. Keeping this broader goal in mind, we focus on extreme value theoretic properties of stationary S$\alpha$S random fields indexed by such groups.

The second motivation comes from the very simple observation that by passing to the Cayley graph of the underlying free group, we obtain a stationary stable random field indexed by a regular tree of even degree. This, of course, is an important object to study (see, for example, \cite{pemantle:1995} for a survey on stochastic processes indexed by trees and their importance in probability theory, statistical physics, fractal geometry, branching models, etc.). To our knowledge, the only family of tree-indexed processes with stable (or even heavy tailed) marginals was introduced by \cite{durrett:1979, durrett:1983} in the form of branching random walks (see also \cite{kyprianou:1999}, \cite{gantert:2000}, and the more recent works of \cite{berard:maillard:2014}, \cite{maillard:2016}, \cite{lalley:shao:2016}), \cite{bhattacharya:hazra:roy:2016a}, \cite{bhattacharya:hazra:roy:2016b}. However, the branching random walks are, by design, highly nonstationary.  In particular, no stationary stable random field has been constructed on a tree so far and our work can perhaps fill in this gap.

An important manifestation of non-amenability of free groups is that the usual ball and its interior boundary are ``asymptotically proportional'' in size. As a result, compared to the case $G=\mathbb{Z}^d$, we indeed observe a different extremal behaviour of $\{X_t\}_{t \in G}$ when $G$ happens to be a finitely generated free group. In \cite{samorodnitsky:2004a, roy:samorodnitsky:2008}, it was shown that a maxima sequence of $\{X_t\}_{t \in \mathbb{Z}^d}$ (obtained by varying $t$ in $d$-dimensional cubes of increasing size) grows faster as we pass from a conservative to a non-conservative $\mathbb{Z}^d$-action in its integral representation. In case of finitely generated free groups, we have observed a phase transition behaviour of a similar maxima sequence and the transition boundary is a different one. In order to confirm the presence of a new dichotomy, we study a class of stable random fields generated by the canonical action of the free group on its Furstenberg-Poisson boundary with the measure being Patterson-Sullivan; see Example~\ref{eg2} below. Even though this nonsingular action is conservative, the maxima of these fields grow as fast as the maxima in the dissipative case.

For stationary \sas random fields generated by dissipative actions of the free group, the corresponding extremal point process has been shown to converge weakly (in the space of Radon point measures on $[-\infty, \infty] \setminus \{0\}$ equipped with the vague topology) to a new kind of point process that we have termed \emph{randomly thinned cluster Poisson process}. This limit too is much more sophisticated compared to the corresponding one in the case $G = \mathbb{Z}^d$ (see \cite{Resnick:Samorodnitsky:2004, Roy:2010}), where a simple cluster Poisson limit was obtained with no thinning. The presence of thinning in our framework can be explained by the nontrivial contributions of the points coming from the boundary of a ball and hence is clearly a ``non-amenable phenomenon''. The asymptotic behaviour of the maxima can easily be read off from the weak convergence of the point process and not surprisingly, the constant term in this limit is much more delicate than the one in the lattice case.

We would like to mention here that the proofs of the main results of this paper are not at all straightforward. The proof of Theorem~\ref{decomp}, for example, relies on the use of ergodic theoretic machineries including \emph{Maharam extension} (see \cite{maharam:1964}) and \emph{measurable union of a hereditary collection} (see \cite{aaronson:1997}), and a combinatorial tool from geometric group theory. On the other hand, the argument used in proving Theorem~\ref{maintheorem} is more probabilistic (and to some extent analytic) in nature. Due to the non-amenability of free groups, even to establish that the limiting point process is Radon, we need to give a sharp bound on an expected value based on exact counting of vertices (see Lemma~\ref{lemma:vertex_counting}) that are a specified distance away from the root and in a certain subgraph of the Cayley tree.

The paper is organized as follows. Section~\ref{BG} is devoted to background information on \sas random fields and their relations to the ergodic theoretic properties of the underlying group actions. In Section \ref{PM}, we present our results on the rate of growth of partial maxima for stationary \sas random fields indexed by general finitely generated countable groups, and in particular by finitely generated free groups. Section \ref{Disscase} deals with the weak convergence of point processes associated with stable fields generated by dissipative actions of finitely generated free groups. The results in Sections~\ref{PM} and~\ref{Disscase} are proved  in Sections~\ref{CP} and~\ref{DP}, respectively.

The following notations are going to be used throughout this paper. For two sequences of positive real numbers $\{a_n\}$ and $\{b_n\}$, the
notation $a_n \sim b_n$ will mean $a_n/b_n \to 1$ as $n \to \infty$. On the other hand, for two $\sigma$-finite measures $m_1$ and $m_2$ defined on the the same measurable space, $m_1 \sim m_2$ will signify that the measures are equivalent. For any $\sigma$-finite measure space $(S,\mathcal{S}, m)$, we define the function space $\mathcal{L}^{\alpha}(S,m):=\left\{f:S\to\mathbb{R} \mbox{ measurable}: \|f\|_\alpha <\infty \right\}$, where
\[\|f\|_\alpha:=\left(\int_S|f(s)|^{\alpha}\,m(ds)\right)^{1/\alpha}.\]
For two random variables $X$, $Y$ (not necessarily defined on the same probability space), the notation $X\stackrel{\text{d}}{=}Y$ indicates that $X$ and $Y$ are identically distributed. For two random fields $\{X_t\}_{t \in G}$ and $\{Y_t\}_{t \in G}$, we write $X_t\stackrel{\text{fdd}}{=}Y_t$, $t \in G$ to mean that they have the same finite-dimensional distributions.

\section{Background} \label{BG}


Let $(G,\cdot)$ be a countable group (which will be a finitely generated free group in most cases) with identity element $e$ and $(S, \mathcal{S}, m)$ be a $\sigma$-finite measure space. A collection of measurable maps $\phi_t : S \rightarrow S$ indexed by $t \in G$ is
called a group action of $G$ on the measurable space $(S, \mathcal{S})$ if
\begin{itemize}
\item[1)] $\phi_e$ is the identity map on $S$, and
\item[2)] $\phi_{u\cdot v} = \phi_v  \circ \phi_u$ for all $u, v \in G$.
\end{itemize}
Note that the order in which the two maps $\phi_v$ and $\phi_u$ appear in the above definition is important because $G$ is mostly going to be a noncommutative free group in this work. A group action $\{\phi_t \}_{t\in G}$ of $G$ on $S$ is called nonsingular if $m \circ \phi_t \sim m$ for all $t \in G$. Here $\sim$ denotes equivalence of measures.

Let $\mathbf{X}=\{X_t\}_{t \in G}$ be an \sas ($0< \alpha <2$) random field indexed by $G$. Any such random field has an integral representation of the type
   \begin{equation}\label{integrep}
   X_t \overset{\text{fdd}}{=} \int_S f_t(s)M(ds), \mbox{ \ \  } t \in G,
   \end{equation}
where $M$ is an \sas random measure on some standard Borel space $(S,\mathcal{S})$ with $\sigma$-finite control measure $m$, and $f_t \in \mathcal{L}^\alpha (m)$ for all $t \in G$. See, for instance, Theorem 13.1.2 of Samorodnitsky and Taqqu (1994) \cite{samorodnitsky:taqqu:1994}. One can assume, without loss of generality, that the union $\bigcup_{t \in G} \mbox{Support}(f_t)$ of the supports of $f_t$ is equal to $S$.

If further $\{X_t\}_{t \in G}$ is stationary, then one can show, following an argument of Rosi{\'n}ski (see \cite{Rosinski:1994}, \cite{Rosinski:1995}, \cite{Rosinski:2000}), that there always exists an integral representation of the following special form
 \begin{equation}\label{eq1}
 f_t(s)=c_t(s)\left(\frac{d\  m\circ\phi_t}{d\ m}(s)\right)^{1/\alpha} f\circ \phi_t(s), \mbox{\ \ } t \in G,
 \end{equation}
 where $f\in \mathcal{L}^{\alpha}(S,m)$, $\{\phi_t\}_{t \in G}$ is a nonsingular $G$-action on $S$, and
$\{c_t\}_{t\in G}$ is a measurable cocycle for $\{\phi_t\}$ taking values in $\{-1,+1\}$ (i.e.,
each $c_t$ is a measurable map $c_t : S \to \{-1,+1\}$ such that for all $t_1, t_2 \in G$, $c_{t_1\cdot t_2}(s) = c_{t_2} (s)c_{t_1}(\phi_{t_2}(s))$ for $m$-almost all $s \in S$). One says that a stationary \sas random field $\{X_t\}_{t\in G}$ is generated by a nonsingular $G$-action $\{\phi_t\}$ if it has an integral representation of the form \eqref{eq1}.

A  measurable set $W \in \mathcal{S}$ is called a wandering set for the action $\{\phi_t \}_{t\in G}$ if $\{\phi_t(W) : t \in G\}$ is a
pairwise disjoint collection. It is a well-known result (see, for example, \cite{aaronson:1997} and \cite{krengel:1985}) that $S = \mathcal{C} \cup \mathcal{D}$, where $\mathcal{C}$ and $\mathcal{D}$ are disjoint and $\{\phi_t\}$-invariant measurable sets
such that
\begin{itemize}
\item[1)] $\mathcal{D}=\cup_{t\in G}\phi_t(W^*)$ for some wandering set $W^*$,
\item[2)] $\mathcal{C}$ has no wandering subset of positive measure.
\end{itemize}
This decomposition of $S$ into two invariant parts is known as the Hopf decomposition. $\mathcal{D}$ is called the \textit{dissipative} part, and $\mathcal{C}$ the \textit{conservative} part of the action, and the corresponding action $\{\phi_t\}$ is called conservative if $S=\mathcal{C}$ and dissipative if $S=\mathcal{D}$.

Another important decomposition is the Neveu decomposition (see, for example, \cite{aaronson:1997}) of $S$ into the \textit{positive} and \textit{null} parts of the nonsingular action as described below. Following Lemma~2.2 and Theorem 2.3~(i) in \cite{Wang:Yizao:Roy:2013} (the arguments in the proof apply to all countable groups, not just $\mathbb{Z}^d$) we decompose $S=\PP\cup \NN$ into two $\{\phi_t\}$-invariant sets $\PP$ (positive part) and $\NN$ (null part), where the set $\PP$ is the largest (modulo $m$) set where one can have a finite measure equivalent to $m$ that is preserved by $\{\phi_t\}$, and $\NN$ is the complement of $\PP$. Obviously $\mathcal{P}\subseteq \mathcal{C}$ because a nontrivial wandering set will never allow a finite invariant measure equivalent to $m$. A measurable subset $B \subseteq S$ is called weakly wandering if there is a countably infinite subset $\{t_n: n \in \N\} \subseteq G$ such that $\phi_{t_n}(B)$ are all disjoint. Clearly the positive part $\PP$ has no weakly wandering set of positive measure.

 Following the notations used in \cite{Rosinski:1995} and \cite{Rosinski:2000}, it is easy to obtain the following unique in law decomposition of the random field $\{X_t\}_{t \in G}$ as
$$ X_t \overset{\text{fdd}}{=} \int_\mathcal{C} f_t(s)M(ds) + \int_\mathcal{D} f_t(s) M(ds) =: X_t^{\mathcal{C}} +X_t ^{\mathcal{D}}, \mbox{\ \ } t \in G$$
into a sum of two independent random fields $X_t^{\mathcal{C}}$ and $X_t ^{\mathcal{D}}$, generated by conservative and dissipative $G$-actions, respectively. Note that, following the same proof as that of Proposition 3.1 in \cite{roy:samorodnitsky:2008}, if a stationary \sas random field $\{X_t\}_{t \in G}$ is generated by a conservative
(dissipative, resp.) $G$-action, then in any other integral representation of $\{X_t\}$ the $G$-action must be conservative (dissipative,
resp.).

Roughly speaking, stable random fields generated by conservative actions tend to have longer memory simply because a conservative action ``keeps coming back". For $G = \mathbb{Z}^d$, this was made precise by studying the rate of growth of partial maxima and limits of sequences of scaled point processes in \cite{samorodnitsky:2004a}, \cite{Resnick:Samorodnitsky:2004}, \cite{roy:samorodnitsky:2008} and \cite{Roy:2010}. We review their results here. Let $\{X_t\}_{t \in \mathbb{Z}^d}$ be a stationary \sas random field and $M_n:= \max\limits_{\|t\|_\infty\leq n}|X_t|$ for $n \geq 1$ with $\|\cdot\|_\infty$ being the $L^\infty$-norm. Then as $n \rightarrow \infty$,
\begin{equation}
n^{-d/\alpha} M_n \Rightarrow \left\{
                                     \begin{array}{ll}
                                     \mathfrak{C}_\alpha^{1/\alpha} \kappa_X  Z_\alpha & \mbox{ if $\{\phi_t\}_{t \in \mathbb{Z}^d}$ is not conservative,} \\
                                     0              & \mbox{ if $\{\phi_t\}_{t \in \mathbb{Z}^d}$ is conservative.}
                                     \end{array}
                              \right. \nonumber 
\end{equation}
Here
\begin{equation}\label{sttail}
 \mathfrak{C}_\alpha =\left(\int_0^\infty x^{-\alpha}\sin x dx \right)^{-1}
      = \left\{
      \begin{array}{ll}
      \frac{1-\alpha}{\Gamma(2-\alpha)\cos(\pi \alpha/2)} & \mbox{if $\alpha \ne 1, $}\\
      \frac{2}{\pi}                                        & \mbox{if $\alpha =1$,}
      \end{array}
      \right.
\end{equation}
$Z_\alpha$ is a standard Frech\'{e}t type extreme value random variable with distribution function
\begin{equation}
P(Z_\alpha \leq x)=e^{-x^{-\alpha}},\;\,x > 0,  \label{Frechet}
\end{equation}
and $\kappa_X$ is a positive constant depending only on the random field $\{X_t\}_{t \in \mathbb{Z}^d}$. In other words, if $\{X_t\}_{t \in \mathbb{Z}^d}$ is generated by a conservative action, then the maxima sequence $M_n$ grows at a slower rate because longer memory prohibits sudden changes in $X_t$ even when $\|t\|_\infty$ is large.

The following result on weak convergence of a sequence of scaled point processes associated with stationary \sas random fields on $\mathbb{Z}^d$ generated by dissipative action is from \cite{Resnick:Samorodnitsky:2004} ($d=1$ case) and \cite{Roy:2010} ($d>1$ case). Assume now that $\{X_t\}_{t \in \mathbb{Z}^d}$ is generated by a dissipative $\mathbb{Z}^d$-action. In this case, we can assume without loss of generality that $\{X_t\}_{t \in \mathbb{Z}^d}$ has the following mixed moving average representation (in the sense of \cite{surgailis:rosinski:mandrekar:cambanis:1993}):
\begin{eqnarray} \label {ma}
X_t \overset{\text{fdd}}{=} \int_{W\times \mathbb{Z}^d}f(w, s-t)M(dw,ds), \;\; t\in \mathbb{Z}^d,
\end{eqnarray}
where $f \in \mathcal{L}^\alpha(W\times \mathbb{Z}^d,\nu\otimes \zeta)$, $\zeta$ is the counting measure on $\mathbb{Z}^d$, $\nu$ is some $\sigma$-finite measure on the standard Borel space $(W,\mathcal{W})$, and $M$ is a \sas random measure on $W\times \mathbb{Z}^d$ with control measure $\nu\otimes \zeta$; see \cite{Rosinski:2000} and \cite{roy:samorodnitsky:2008} for details.

Suppose $\nu_\alpha$ is the symmetric measure on $[-\infty,\infty] \setminus \{0\}$ such that $\nu_\alpha(x,\infty] = \nu_\alpha [-\infty,-x) = x^{-\alpha}$ for all $x > 0$. Let
\begin{equation}\label{PRM}
\sum_{i}\delta_{(j_i,w_i,u_i)} \sim \mbox{PRM}(\nu_\alpha \otimes \nu \otimes \zeta)
\end{equation}
be a Poisson random measure on $\big([-\infty,\infty] \setminus \{0\}\big)\times W \times \mathbb{Z}^d$ with mean measure $\nu_\alpha \otimes \nu \otimes \zeta$. Then $\{X_t\}_{t \in \mathbb{Z}^d}$ in \eqref{ma} has the following series representation (ignoring a factor of $\mathfrak{C}_\alpha^{1/\alpha}$):
\begin{equation}\label{series}
 X_t \overset{\text{fdd}}=  \sum_i j_if(w_i,u_i-t), \quad t \in \mathbb{Z}^d,
\end{equation}
It was shown in \cite{Resnick:Samorodnitsky:2004} and \cite{Roy:2010} that in the space $\mathcal{M}$ of Radon measures on $[-\infty,\infty] \setminus \{0\}$ (endowed with vague topology),
$$\sum_{\|t\|_\infty \leq n} \delta_{(2n)^{-d/\alpha}X_t}\Rightarrow \tilde{N}_*,$$
which is a cluster Poisson random measure with representation
\begin{equation}{\label{NstarZ}}
  \tilde{N}_*  =  \sum_{i=1}^\infty \sum_{t \in \mathbb{Z}^d}\delta_{j_if(w_i, t)}
      \mathbf{1}_{(u_i=\mathbf{0})},
\end{equation}
where $j_i, w_i, u_i$ are as in \eqref{PRM}. The Laplace functional of the above $\tilde{N}_*$ is
$$\E\big(e^{-\tilde{N}_*(g)}\big)=\exp\left\{-\int_W \int_{|x|>0}\bigg(1-\exp{\bigg\{-\sum_{t\in \mathbb{Z}^d}g(x(f(w, t)))\bigg\}}\bigg)\nu_\alpha(dx)\nu(dw)\right\},$$
for all measurable $g: [-\infty,\infty] \setminus \{0\} \to [0, \infty)$. Here $\tilde{N}_*(g)$ denotes the random variable obtained by integrating $g$ with respect to the random measure $\tilde{N}_*$. Note that in the representation of the cluster Poisson random measure $N_*$ given in Theorem 3.1 in \cite{Roy:2010}, the term $\mathbf{1}_{(u_i=\mathbf{0})}$ was missing even though the computation of the limiting Laplace functional was correct. A similar comment applies to Theorem 3.1 of \cite{Resnick:Samorodnitsky:2004}.


Recall that for any finitely generated countable group $G$ with a symmetric (w.r.t. taking inverses) generating set $D$ not containing the identity element $e$, the Cayley graph $(V, E)$ consists of the vertex set $V=G$ and edge set $E=\{(u,v): u^{-1}v \in D\}$. Clearly, symmetry of $D$ turns this into an undirected graph and $e \notin D$ implies there is no self-loop. In this paper, we shall use the language of Cayley graphs to investigate the asymptotic behaviours of a sequence of partial maxima and a sequence of point processes associated with the stationary \sas random fields indexed by finitely generated free groups. In most of the discussions below, $G$ will denote a free group of finite rank $d \geq 2$ (except in Theorem \ref{decomp}, where $G$ will simply be a general finitely generated countable group) with the generating set $D=\{a_1, a_1^{-1}, a_2, a_2^{-1}, \ldots, a_d, a_d^{-1}\}$ being the collection of $d$ independent symbols and their inverses. This group consists of all reduced words formed out of the symbols in $D$ with the operation being ``concatenation followed by reduction'' and its Cayley graph is a $2d$-regular tree. See, for example, \cite{aluffi:2009} for details on free groups and Cayley graphs.

For any $t\in G$, we define $|t|$ to be the graph distance of $t$ from the root $e$ in the Cayley graph of the group $G$, i.e.,
$|t|=d(v,e)$, where $d(a,b)$ denotes the graph distance between vertices $a$ and $b$ in the Cayley graph of $G$.
Also
\begin{equation}\label{ballboundary}
\begin{array}{rl}
E_n&:=\{t\in G: |t| \leq n\}, \mbox{\ \ \ and} \\
C_n&:=\{t \in G: |t|=n\}
\end{array}
\end{equation}
denote the ball of radius $n$ and its interior boundary, respectively. When $G$ is a free group of finite rank $d \geq 2$, an easy counting yields that
\[|E_n|=1+\frac{d}{d-1}\left[(2d-1)^n-1\right] =\Theta((2d-1)^n)\]
and $|C_n|= (2d)(2d-1)^{n-1}$  for all $n \geq 1$. In particular, $|C_n|$ is ``asymptotically proportional'' to $|E_n|$, which is a manifestation of non-amenability. As a result, the extreme values of stable random fields indexed by finitely generated free groups are affected by the significant contributions from the interior boundary of $E_n$. This will become clear in Sections~$\ref{PM}$ and $\ref{Disscase}$ below.

In the next section, we shall study the asymptotic behaviour of the partial maxima sequence of
\begin{equation}
M_n=\max _{t\in E_n}|X_t|, \; n \geq 1 \label{def_M_n_general}
\end{equation}
of the stationary \sas random field $\{X_t\}_{t \in G}$ obtained by restricting the field to the ball $E_n$. As we shall see, there will be a phase transition as long as $G$ is a finitely generated countable group. Of course, for $G=\mathbb{Z}^d$, the phase transition boundary has to coincide with the Hopf boundary. However, when $G$ is a free group of finite rank $d \geq 2$, non-amenability of the group will induce a new transition boundary that lies strictly between the Hopf and Neveu boundaries.

\section{Rate of growth of partial maxima}\label{PM}

Let $G$ be a countable group generated by a finite symmetric set $D$ and $\{X_t\}_{t\in G}$ be a stationary \sas random field having an integral representation of the form \eqref{integrep}, where $f_t$ is given by \eqref{eq1}. We shall eventually specialize to the case when $G$ is a free group of finite rank $d \geq 2$ and investigate the extreme value theory of the field. Define $E_n$ and $C_n$ as in \eqref{ballboundary} and the partial maxima sequence $M_n$ by \eqref{def_M_n_general}. We define a deterministic sequence
\begin{equation}\label{definitionofb_n}
b_n = b_n(f) = \left(\int \max _{t\in E_n}|f_t(x)|^\alpha m(dx)\right)^{1/\alpha} , \mbox{\ \ \ \ \ \ } n=1,2,\ldots\, ,
\end{equation}
where $f \in \mathcal{L}^\alpha(S,m)$ is used in the definition of $f_t$ in \eqref{eq1}. Note that by Corollary~4.4.6 of \cite{samorodnitsky:taqqu:1994}, for any specific random field $\{X_t\}_{t\in G}$, the quantity $b_n$ does not depend on the choice of $f_t$ in its integral representation \eqref{integrep}. However, in this article, we shall analyze a class of stationary \sas random fields obtained by varying $f \in \LLL^\alpha(S, m)$, and fixing the group action $\{\phi_t\}_{t \in G}$ and the cocycle $\{c_t\}_{t \in G}$ in \eqref{eq1}. With this viewpoint in mind, we are introducing the notation $b_n(f)$ even though in many situations, we shall stick to $b_n$.

Following arguments similar to that in \cite{samorodnitsky:2004a}, one can show that to a large extent, the asymptotic behaviour of the random sequence $M_n$ is determined by that of $b_n$. Hence we first look at the growth rate of the deterministic sequence $b_n$ and use that to analyze the same for $M_n$. In the sequel, $c$ will always denote a positive constant that may not necessarily be the same in each occurrence.

\begin{thm}\label{decomp} Let $G$ be a countable group generated by a finite symmetric set, and $\{\phi_t\}_{t\in G}$ be a nonsingular group action on a $\sigma$-finite standard measure space $(S, \mathcal{S}, m)$.
\begin{itemize}[leftmargin=*]
\item[(i)] Then the set $S$ can be uniquely decomposed into two disjoint $\{\phi_t\}$-invariant measurable sets $\mathcal{A}$ and $\mathcal{B}$ (i.e. $S=\mathcal{A}\cup \mathcal{B}$) such that, for any $f \in \mathcal{L}^\alpha(S,m)$,
\begin{itemize}
\item[(a)] whenever $f$ is supported on $\mathcal{B}$, $\lim_{n \rightarrow \infty} \frac{b_n(f)}{|E_n|^{1/\alpha}}=0$, and
\item[(b)] if the support of $f$ has some nontrivial intersection with $\mathcal{A}$, then \\
$\limsup_{n \rightarrow \infty} \frac{b_n(f)}{|E_n|^{1/\alpha}}>0$.
\end{itemize}
The above decomposition is the same for all measures equivalent to $m$.
\item[(ii)] The dissipative part $\mathcal{D}\subseteq \mathcal{A}$, and the positive part $\mathcal{P}\subseteq \mathcal{B}$.
\item[(iii)] If the \sas random field is given by the integral representation \eqref{integrep} and \eqref{eq1}, then
$$X_t \overset{\text{fdd}}{=} \int_\AA f_t(s)M(ds) + \int_\BB f_t(s) M(ds)=:X_t^\AA + X_t^\BB, \;t\in G$$
can be written as a sum of independent random fields $X_t^\AA$ and $X_t^\BB$ such that the following results hold.
\begin{itemize}
\item[(a)] If the component $X_t^\AA$ is zero, then
$$M_n/|E_n|^{1/\alpha} \overset{\P}{\rightarrow}0$$
\item[(b)] If the component $X_t^\AA$ is nonzero, then $M_n = O_p(|E_n|^{1/\alpha})$ (i.e., $M_n/|E_n|^{1/\alpha}$ is tight), and there exists a subsequence $M_{n_k}$ of $M_n$ and a positive constant $c>0$ such that $$M_{n_k}/|E_{n_k}|^{1/\alpha}\Rightarrow cZ_\alpha,$$ where $Z_\alpha$ is an $\alpha$-Fr\'{e}chet random variable with distribution function given in \eqref{Frechet}.
\end{itemize}
\end{itemize}
\end{thm}

Keeping in mind the first part of the above theorem, we shall call $\AA$ the nondegenerate part and $\BB$ the degenerate part. It is possible that this decomposition may be known in the ergodic theory literature by some other name although our extensive literature search did not reveal any. When $G=\Z^d$, the above decomposition is the same as the Hopf decomposition of the group action with $\AA=\mathcal{D}$ and $\BB=\mathcal{C}$ (see \cite{samorodnitsky:2004a} and \cite{roy:samorodnitsky:2008}). For a general group $G$, even if the support of $f$ has a nontrivial intersection with the nondegenerate part $\AA$, one cannot surely say that $\lim_{n \rightarrow \infty} \frac{b_n(f)}{|E_n|^{1/\alpha}}>0$ simply because the limit may not always exist. In particular, we need to work with  limit superior as opposed to the limit in Part (i)(b) of Theorem~\ref{decomp}. However, when $G$ is a free group of finite rank $d \geq 2$, we can significantly improve our previous result as shown in the  following theorem.

\begin{thm}\label{freegroupliminf} When $G$ is a free group of finite rank $d \geq 2$, for any $f\in \LLL^\alpha(S,m)$ whose support has some nontrivial intersection with $\AA$, one has,
$$\liminf_{n \rightarrow \infty} \frac{b_n(f)}{|E_n|^{1/\alpha}}>0.$$
Also given any subsequence $M_{n_k}$ of $M_n$, there exists a further subsequence $M_{n_{k_\l}}$ and a positive constant $c>0$ such that
$$M_{n_{k_\l}}/|E_{n_{k_\l}}|^{1/\alpha}\Rightarrow cZ_\alpha,$$
where $Z_\alpha$ is an $\alpha$-Fr\'echet random variable as before.
\end{thm}

In fact, if $f$ is supported on the dissipative part $\DD$, then for any finitely generated countable group $G$, $\liminf_{n \rightarrow \infty} \frac{b_n(f)}{|E_n|^{1/\alpha}}>0$ (see the proof of Part (ii) of Theorem \ref{decomp}). When $G$ is a free group of finite rank $d \geq 2$ and Support$(f) \subseteq \DD$, then the limit exists and as a consequence, $M_n/(2d-1)^{n/\alpha} \Rightarrow c Z_\alpha$ for some $c>0$; see Corollary~\ref{maxima} below. For the rest of this section and the next one, we shall assume that $G$ is a free group of finite rank $d \geq 2$. Thanks to the non-amenability of this group, the decomposition of $S$ into degenerate and nondegenerate parts is now different from what happens in the  $\Z^d$ case, where it coincides with the Hopf decomposition. This leads to a new dichotomy (see below) for the maxima sequence $M_n$ defined in \eqref{def_M_n_general}.

\begin{thm}\label{shiftofboundary}
When $G$ is a free group of finite rank $d \geq 2$, there exists a stationary \sas random field indexed by $G$ generated by a conservative action, for which we have $M_n/(2d-1)^{n/\alpha} \Rightarrow \mathfrak{C}_\alpha^{1/\alpha} Z_\alpha$, where $\mathfrak{C}_\alpha$ is as defined in \eqref{sttail} and $Z_\alpha$ is a standard $\alpha$-Fr\'{e}chet random variable. Moreover if $\CC\NN:=\CC\cap \NN$ denotes the conservative null part of the action, then $\CC\NN$ can have nontrivial intersections with both the nondegenerate part $\AA$ and the degenerate part $\BB$.
\end{thm}
That is, we shall give two instances (see Examples~\ref{eg2} and \ref{eg3} below) of stationary \sas random fields generated by conservative null actions, such that for one, the partial maxima grows at the rate of $(2d-1)^{n/\alpha}$ (or $|E_n|^{1/\alpha}$) and for the other, the partial maxima grows at a strictly smaller rate.
\begin{figure}[h]
    \centering
    \includegraphics[width=0.35\textwidth]{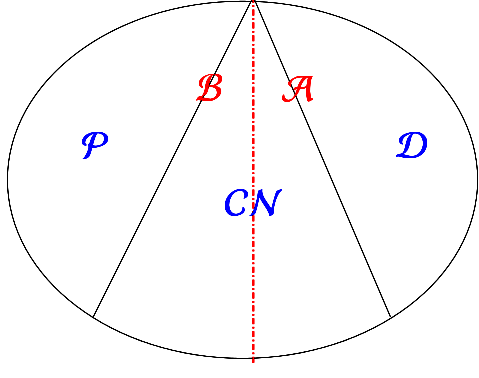}
    \caption{Boundary between nondegenerate part ($\mathcal{A}$) and degenerate part ($\mathcal{B}$)}
    \label{fig:boundary}
\end{figure}
Note that Hopf and Neveu decompositions of the underlying nonsingular action induce the partition of $S = \mathcal{P} \cup \CC\NN \cup \DD$ into positive, conservative null, and dissipative parts. Our phase transition boundary (between the degenerate and the non-degenerate parts) lies strictly between the Hopf and Neveu boundaries and passes through the conservative null part ($\CC\NN$) of the group action; see the dotted line in Figure~\ref{fig:boundary}.

The next result says that the asymptotic behaviour of the partial maxima for the balls of increasing radii is actually determined by the interior boundaries of the balls. Clearly, this is intrinsically a non-amenable phenomenon that would never happen in the lattice case.

\begin{thm}\label{boundarymatters} Let $G$ be a free group of finite rank $d\geq 2$, and let the stationary \sas random field indexed by $G$ has integral representation \eqref{integrep}. Then we have
\begin{eqnarray*}
\limsup_{n\rightarrow \infty} \frac{\int \max_{t\in E_n}|f_t(x)|^\alpha m(dx)}{(2d-1)^{n}}>0 \mbox{\ \ \ \ if and only if\ \ \ \ \ } \\
\limsup_{n\rightarrow \infty} \frac{\int \max_{t\in C_n}|f_t(x)|^\alpha m(dx)}{(2d-1)^{n}}>0.
\end{eqnarray*}
\end{thm}

In the next theorem, we try to find some sets that belong to the nondegenerate part $\AA$ of Theorem \ref{decomp}. It states that if a set has sufficient number of disjoint translates in each ball, then the set is inside $\AA$.

\begin{thm} \label{setsinA}
Define, for any subset $B \subseteq S$, $a_n(B)$ to be the maximum number of sets in $\{\phi_t(B): t\in E_n\}$ that are pairwise disjoint, i.e., $a_n(B):= \max\{|T|: T \subseteq E_n \mbox{\and\ \ } \phi_t(B) \mbox{\ are pairwise disjoint for all\ \ } t \in T\}$. If $\limsup_{n\rightarrow \infty} \frac{a_n(B)}{|E_n|}>0$ for some subset $B \subseteq S$, then $B \subseteq \AA$.
\end{thm}

The proofs of the theorems stated in this section are given in Section \ref{CP}. Finally, we give three examples of stationary \sas random fields generated by conservative actions. The first example holds for any countable finitely generated group $G$, and is crucial for the proof of Part (iii) of Theorem \ref{decomp}. This is parallel to Example 5.4 in \cite{samorodnitsky:2004a}.

\begin{example}\label{eg1} Let $S=\R^G$ and $M$ is an \sas random measure on $\R^G$ whose control measure $m$ is a probability measure under which the projections $(\pi_t, t\in G)$ are i.i.d. random variables with a finite absolute $\alpha^{\mbox{th}}$ moment. Let $\pi=\pi_e:\R^G\mapsto \R$ as $\pi((x_t)_{t\in G})=x_e$, and $\phi_t$ is the shift operator, i.e., $(\phi_t((x_s)_{s\in G}))_k=x_{t\cdot k}$. Clearly this action is probability $m$-preserving and hence is conservative. The random field has the integral representation
\[
X_t\overset{\text{fdd}}{=}\int_{\R^G}\pi\circ \phi_t dM=\int_{\R^G}\pi_tdM, \quad t\in G.
\]
Now, if the projections $\pi_t$, $t\in G$ are i.i.d. Pareto random variables with
$m(\pi_e>x)=x^{-\theta}$, $x\geq 1$ for some $\theta >\alpha$, then as in Example 5.4 in \cite{samorodnitsky:2004a}, we get,
$$b_n\sim c_{\alpha,\theta}^{1/\alpha}|E_n|^{1/\theta} \quad \mbox{as\ } n\rightarrow \infty $$
for some positive constant $c_{\alpha,\theta}$. Furthermore, in this case, $M_n/|E_n|^{1/\theta}$ converges to an $\alpha$-Fr\'echet distribution. This example will be required in the proof of Part~(iii) of Theorem \ref{decomp}.
\end{example}
In the next two examples, $G$ is a free group of finite rank $d \geq 2$. The first one considers the canonical action of the free group on its Furstenberg-Poisson boundary with the Patterson-Sullivan measure on it. Any stationary \sas random field generated by this nonsingular action satisfies $\liminf_{n\rightarrow \infty} b_n/|E_n|^{1/\alpha}>0$ even though the action is conservative.

\begin{example}\label{eg2} The boundary $\d G$ of the group $G$ consists of all infinite length reduced words made of powers of symbols from the generating set $D$.  Given a group element $g\in G\setminus \{e\}$, define $H_g (\subset \d G)$ to be the cylinder set consisting of all infinite words starting with $g$, i.e.,
\[
H_g=\{\omega\in \d G:[\omega]_{|g|}=g\},
\]
where $[\omega]_n$ represents the element in $G$ formed by the first $n$-length segment of $\omega$. Define $\mathcal{S}$ to be the $\sigma$-field on $S=\d G$ generated by the cylinder sets $H_g$, $g \in G \setminus \{e\}$. It is easy to check that there exists unique probability measure $m$ on $(S, \mathcal{S})$ such that $$m(H_g)=\frac{1}{2d(2d-1)^{|g|-1}} \quad \mbox{for all \ \ } g\in G\setminus e.$$
This measure is known as the Patterson Sullivan measure (see \cite{quint:2006}) and it turns $S=\d G$ into a Furstenberg-Poisson boundary (see \cite{tamuz:2016}) of the group $G$.

The free group $G$ acts canonically on $(S, \mathcal{S}, m)$ in a nonsingular fashion by
\begin{equation}\label{bdryactn}
\phi_t(\omega)=t^{-1} \cdot \omega, \quad \mbox{for \ \ } t\in G,\,\omega \in S,
\end{equation}
where $\cdot$ is the left-concatenation of a finite word with an infinite word followed by reduction. The Radon-Nikodym derivatives of this action are given by
$$\frac{dm\circ \phi_t}{dm}(\omega)=(2d-1)^{-B_\omega(t)}, \quad t\in G,\,\omega \in S,$$
where $B_\omega (t)=|t|-2|t\wedge \omega|$ (the \emph{Busemann function} associated with $\omega$) with $t\wedge \omega$ being the \emph{longest common initial segment} (also known as the \emph{confluent}) of $t$ and $\omega$. For further details on the boundary action, we refer the reader to \cite{Grigorchuk:2012}, where it was established that this action is conservative.

We shall first show that the boundary action is null, i.e., its positive part (in the Neveu decomposition) is empty modulo $m$. To this end, consider the generating set $D=\{a_1,a_1^{-1},a_2,a_2^{-1},\ldots,a_d,a_d^{-1}\}$ of $G$ with $d$ independent symbols and their inverses as before. Then take the cylinder set $H_{a_1^{-1}}$ and consider the sets $\phi_t(H_{a_1^{-1}})$ for $t=e$ and $t=a_1^{-1}x a_1^{-k}$, where $k=0,1,2,\ldots$ and $x \in D\setminus \{a_1,a_1^{-1}\}$. It is easy to see that, for all such $t$, $\phi_t(H_{a_1^{-1}})$ are disjoint and their union is $\d G$. This shows that  $H_{a_1^{-1}}$ is a weakly wandering set, and hence the boundary action is null.

Define $f_t$ by \eqref{eq1} with the constant function $f\equiv 1$ on S, the trivial cocycle $c_t \equiv 1$ for all $t \in G$,  and the boundary action \eqref{bdryactn}. Then $\{X_t\}_{t \in G}$ defined by the integral representation \eqref{integrep} is a stationary \sas random field generated by a conservative null action. We now claim that
\begin{equation}
F_n(\omega):=\max_{t\in E_n}|f_t(\omega)|^\alpha=\max_{t\in E_n}\frac{dm\circ \phi_t}{dm}(\omega)=\max_{t\in E_n}(2d-1)^{-B_\omega(t)}=(2d-1)^{n}\label{eq:series_F_n}
\end{equation}
for all $\omega \in S$. We need to show the last equality above. As $t\in E_n$, for any $\omega \in S$, the length of the confluent $|t\wedge \omega| \leq |t| \leq n$. Hence
$$-n\leq B_\omega(t) \leq n \quad \mbox{for all}\quad t\in E_n,\,\omega \in S.$$
Hence $F_n(\omega)\leq (2d-1)^n$. To see the other inequality, note that, for any fixed $\omega \in  S=\d G$, if we take $g=[\omega]_n\in C_n \subset E_n$, then $|g\wedge \omega|=n$, so that
$$B_\omega (g)=|g|-2|g\wedge \omega|=n-2n=-n.$$
Hence, $(2d-1)^{-B_\omega(g)}=(2d-1)^n$, so that $F_n(\omega)\geq (2d-1)^n$. This proves \eqref{eq:series_F_n}.

Hence,
\[(b_n(f))^\alpha=\int_{S} F_n(\omega) m(d\omega)=(2d-1)^n, \quad \mbox{and,}\]
\[\lim_{n\rightarrow \infty} \frac{(b_n(f))^\alpha}{|E_n|}=\lim_{n\rightarrow \infty} \frac{(b_n(f))^\alpha}{\frac{d}{d-1}(2d-1)^n}=\frac{d-1}{d}>0.\]
Following the arguments in the proof of Theorem 4.1 in \cite{samorodnitsky:2004a}, one has
$$\frac{M_n}{(2d-1)^{n/\alpha}}\Rightarrow \mathfrak{C}_\alpha^{1/\alpha}Z_\alpha,$$ where $Z_\alpha$ is a standard $\alpha$-Fr\'{e}chet random variable defined in \eqref{Frechet}.

Our next claim is that the degenerate part $\BB$ of the boundary action is an $m$-null set. To establish this, take the cylinder set $H_a$ for any $a$ in the generating set $D$, and the set $T_n=\{a\cdot g : g\in C_{n-1}\}$. Clearly $T_n\subset E_n$ as $|a\cdot g|\leq |a|+|g|=n$, and
\[\{\phi_t(H_a)\}_{t\in T_n}=\{H_{g^{-1}a^{-1}a}:g\in C_{n-1}\}=\{H_g:g\in C_{n-1}\}\] are all the cylinder sets of ``dimension'' $n-1$ which are all disjoint. Since $|T_n|=|C_{n-1}|\sim \frac{2d-2}{(2d-1)^{2}}|E_n|$ as $n \to \infty$, by Theorem~\ref{setsinA}, we have, $H_a \subseteq \AA$. As this happens for all symbols $a\in D$, and the union of the cylinder sets over all $a\in D$ is $S$, one gets that the nondegerate part $\AA=S$ modulo $m$.
\end{example}
\begin{remark} Note that the boundary action defined here differs slightly from that defined in \cite{Grigorchuk:2012}, where the authors define
\[\phi_t(\omega)=t \cdot \omega, \quad \mbox{for \ \ } t\in G,\,\omega \in S.\]
We use the definition in \eqref{bdryactn} so as to match with our convention for group actions used in this paper, i.e, $\phi_{u\cdot v} = \phi_v  \circ \phi_u$ for all $u, v \in G$. This adjustment does change the Radon-Nikodym derivatives but does not compromise the nonsingulatrity (or the conservativity) of the action.
\end{remark}

The above example shows that there exist stationary \sas random fields generated by conservative null actions, for which the maxima grows at the rate of $(2d-1)^{n/\alpha}$. But this is not necessarily the case for all such actions. The next example shows that there exists a stationary \sas random field generated by a null conservative action, for which the maxima sequence grows at a strictly smaller rate.
\begin{example}\label{eg3} Let the free group $G$ of rank $d \geq 2$ be generated by the set $D$ as in Example~\ref{eg2}.  Take $S=\R$ with $m=$ Lebesgue measure, and the group action $\{\phi_t\}_{t \in G}$ to be the one that makes a shift of $1$ by the action of $a_1$ and is fixed by the actions of $a_2, a_3, \ldots a_d$. In other words, for all $i =1, 2, \ldots, d$,
\[
\phi_{a_i}(x) = x + \ind_{\{i=1\}}, \; x \in \mathbb{R}.
\]
This group action is clearly measure preserving. Therefore one has $(b_n(\ind_{(0,1]}))^\alpha=\mbox{Leb}((-n, n+1])=2n+1$,
and hence
$(b_n(\ind_{(0,1]}))^\alpha/|E_n|\rightarrow 0$ as $n\rightarrow \infty$.

Again, the set $(0,1]$ is weakly wandering, as $\phi_t((0,1])$ for $t=a_1^k$, $k\in \Z$ are all disjoint, and their union is the whole set $\R$. As the set $(0,1] \subseteq \BB$ (the degenerate part - recall Theorem \ref{decomp}) and $\BB$ is $\phi_t$-invariant, it contains all translates $\{\phi_t((0,1]),t=a_1^k\}$, and hence $\BB=\R$. Hence this action is conservative, null and yet degenerate. By Theorem~\ref{shiftofboundary}, for any stationary \sas random field generated by this action, the partial maxima satisfies $M_n/(2d-1)^{n/\alpha} \overset{\P}{\rightarrow} 0$.
\end{example}

\section{Dissipative case: point process and maxima}\label{Disscase}

We would like to begin this section by observing that the representations \eqref{ma} and \eqref{series} can be generalized to any countable group $G$, not just $\mathbb{Z}^d$. More specifically, one can establish that for any countable group $G$, a stationary \sas random field $\{X_t\}_{t \in G}$ is generated by a dissipative $G$-action if and only if it has a mixed moving average representation of the form
\begin{eqnarray}\label{ourma}
X_t \overset{\text{fdd}}{=}  \int_{W\times G} f(w,t^{-1}s)dM(w,s), \; t \in G,
\end{eqnarray}
where $M$ is an \sas random measure on $W \times G$ with control measure $\nu \otimes \zeta $, and $\nu$ is a $\sigma$-finite measure on the measurable space $(W,\mathcal{W})$, $\zeta$ is the counting measure on the group $G$, and $f \in \mathcal{L}^\alpha(W\times G,\nu\otimes \zeta)$ (as mentioned in Section~\ref{BG}, this terminology was introduced in \cite{surgailis:rosinski:mandrekar:cambanis:1993}). See \cite{roy:2008}, where the argument is given for any countable abelian group extending the works of \cite{Rosinski:2000} and \cite{roy:samorodnitsky:2008}. With a little bit of care (about the side of multiplication, etc.), such an argument can be carried forward to any countable group, not necessarily abelian. As in Section~\ref{BG}, taking $\nu_\alpha$ as the symmetric measure on $[-\infty,\infty]\setminus\{0\}$ satisfying $\nu_\alpha(x,\infty]=\nu_\alpha[-\infty, -x)=x^{-\alpha}$ for all $x>0$,
\begin{equation}\label{defofPRM}
N =\sum_{i}\delta_{(j_i,v_i,u_i)} \sim \mbox{PRM}(v_\alpha \otimes \nu \otimes \zeta )
\end{equation}
on $\big([-\infty,\infty]\setminus\{0\}\big)\times W\times G$, and dropping a factor of $\mathfrak{C}_\alpha^{1/\alpha}$, one can obtain the series representation
\begin{equation}\label{ourseries}
X_t \overset{\text{fdd}}{=} \sum_i j_if(v_i,t^{-1}u_i), \; t \in G.
\end{equation}

In this section, we shall assume that $G$ is a free group of finite rank $d \geq 2$ and study the weak limit of scaled point process and partial maxima sequences induced by a stationary \sas random field \eqref{ourseries} generated by a dissipative (and hence nondegenerate by Part (ii) of Theorem~\ref{decomp} above) action. Thanks to the nontrivial contributions (see, for instance, Theorem~\ref{boundarymatters} above) coming from the interior boundary $C_n$ of $E_n$ as a result of the non-amenability of $G$, these limits are different from those arising in the case of $\Z^d$. The class of point process limits that we obtain are completely novel and we have termed this new class as \emph{randomly thinned cluster Poisson processes}. We would like to mention once more that in case of $\Z^d$, Poisson cluster processes arise as limits and the ``random thinning'' phenomenon is absent; see \cite{Resnick:Samorodnitsky:2004}, \cite{Roy:2010}. We will state our results for the random field \eqref{ourseries} after defining various quantities that appear in the statement of the main theorem of this section.

\subsection{Construction of $\l$-subgraphs}

For each fixed $\l \in \Z$, we define a class of subgraphs of the Cayley graph of $G$ by specifying the set of vertices of each subgraph. We call them $\l$-subgraphs, and denote the set of all $\l$-subgraphs by $\Gamma_{\l}$. We shall consider three cases and in each case, we shall construct a typical $\l$-subgraph as described below. Recall that for $u,v\in G$, $d(u,v)$ denotes the graph distance between the vertices $u$ and $v$ in the Cayley graph of $G$, $v = d(v, e)$, and $C_n$ denotes the interior boundary of the ball $E_n$ of size $n$.

\noindent \textbf{Case 1: $\l=0$.}  Consider a self-avoiding path starting from the root $e$. Let the vertices along the path be $v_0=e, v_1, v_2,\ldots$, where $|v_k|=k$. For each such vertex $v_k$, we define a collection of sets of vertices $V_k$ by
\begin{equation}
V_k=\{t\in G: d(t,v_k) \leq k\}, \mbox{\ \ \ \ \ } k =0,1,2,\ldots\,. \label{defn:V_k}
\end{equation}
Note that $\{e\} =V_0\subset V_1 \subset V_2 \subset \ldots\, .$ A typical $\l$-subgraph (for $\l=0$) corresponding to a particular self-avoiding path $\{v_0=e, v_1, v_2,\ldots\}$ is defined as the union of all these sets of vertices $\cup_{i=0}^\infty V_i$. The collection of all such subgraphs corresponding to all self avoiding paths starting from the root $e$ is the set $\Gamma_0$.

\noindent \textbf{Case 2: $\l>0$.} Here we consider all self avoiding paths $ \{v_0, v_1, v_2,\ldots\}$ starting from some vertex $v_0\in C_{\l}$ that ``goes away from the root", i.e., $|v_0|=\l, |v_1|=\l+1,  |v_2|=\l+2$ and so on. For any such self avoiding path, define the collection of vertices $V_k$ by \eqref{defn:V_k}, and we have the corresponding typical $\l$-subgraph as $\cup_{i=0}^\infty V_i$. The collection of all such subgraphs is denoted by $\Gamma_\l$.


\noindent \textbf{Case 3: $\l<0$.}  For any fixed $\l <0$, consider all self avoiding paths $ \{v_0, v_1, v_2,\ldots\}$ starting from some vertex $v_0\in C_{|\l|}$ such that $|v_0|=|\l|, |v_1|=|\l|-1, |v_2|=|\l|-2, \ldots, v_{|\l|}=e, |v_{|\l|+1}|=1, |v_{|\l|+2}|=2$ and so on. Given such a path, we define, once again, the corresponding $\l$-subgraph to be $\cup_{i=0}^\infty V_i$, where $V_k$ is as in \eqref{defn:V_k}. The collection of all such subgraphs corresponding to all self avoiding paths is our $\Gamma_\l$.

Given $g \in \mathcal{L}^\alpha(W\times G,\nu\otimes \zeta)$, $\l \in \Z$ and $\xi \in \Gamma_\l$, we define functions $\tilde{g}^{(\l,\xi)}$ by appropriately thinning the function $g$ to the $\l$-subgraph $\xi \in \Gamma_\l$, i.e.,
\begin{equation}
\tilde{g}^{(\l,\xi)}(w,t)=g(w,t)\mathds{1}_{\{t\in \xi\}}, \; w \in W,\, t \in G. \label{defn:thin:g}
\end{equation}

\subsection{An all-encompassing Poisson random measure} Next we shall describe for each $\l \in \mathbb{Z}$, a probability measure $\gamma_\l$ on the set $\Gamma_\l$ of all $\l$-subgraphs as a ``uniform measure on all $\l$-subgraphs''. We shall construct these by resorting to Kolmogorov consistency theorem. To this end, first fix $\l \in \Z$. For any $m \in \N$, we say that two $\l$-subgraphs are \emph{$m$-essentially distinct} if the two subgraphs when restricted to $E_m$ are distinct. We denote, by $\Gamma_\l^{(m)}$, the finite set of all $m$-essentially distinct $\l$-subgraphs.

Define $X=\{1,2,\ldots,2d\}$. We claim that for each $(\l, m) \in \mathbb{Z} \times \mathbb{N}$, the set $\Gamma_\l^{(m)}$ can be embedded into $C_{|\l|} \times X^{m-\l-1}$. To see this, note that any two essentially distinct subgraphs in $\Gamma_\l^{(m)}$ will necessarily correspond to two distinct (self avoiding) paths of length $m-\l$ starting from some vertex in $C_{|\l|}$. (But the path associated to such a subgraph may not be unique, in that case, we just choose any one of the associated paths. However, for any two distinct subgraphs, any two corresponding paths associated to them will necessarily be distinct.) And since the degree of each vertex in $G$ is $2d$, any such path is an element of $C_{|\l|} \times X^{m-\l-1}$. Similarly, $\Gamma_\l$ can be embedded into $C_{|\l|} \times X^\infty$.

Once again, fix $\l \in \mathbb{Z}$. Now suppose that $\gamma_\l^{(m)}$ is the uniform distribution on $\Gamma_\l^{(m)}$ (embedded in $C_{|\l|} \times X^{m-\l-1}$). Then clearly $\{\gamma_\l^{(m)}\}_{m \geq 1}$ is a consistent system of probability measures. Therefore by Kolmogorov consistency theorem, we get a unique probability measure $\gamma_\l$ on $\Gamma_\l$ (embedded in $C_{|\l|} \times X^\infty$), such that $\gamma_\l $ restricted to $\Gamma_\l^{(m)}$ is $\gamma_\l^{(m)}$ for each $m \in \mathbb{N}$. Now that we have defined the sets of $\l$-subgraphs $\Gamma_\l$ and the measures $\gamma_\l$ on them, we consider the product probability space
$$\left(\Gamma=\prod_{\l \in \mathbb{Z}} \Gamma_\l, \, \gamma = \bigotimes_{\l \in \mathbb{Z}} \gamma_\l \right).$$
And as each $\Gamma_\l$ is embedded in a compact separable metric space, so is their product $\Gamma$. In particular, $\Gamma$ is locally compact and separable.

We define a sequence of i.i.d.\ $\Gamma$-valued random variables $\mathbf{r}_i=\big(r_{i,\l}: \l \in \mathbb{Z}\big)$, $i \in \mathbb{N}$ with common law $\gamma$ and independent of the Poisson point process $N$ defined in \eqref{defofPRM}.  We also take a collection of i.i.d.\  integer-valued random variables $s_i$, $i \in \mathbb{N}$ independent of $N$ and $\{\mathbf{r}_i\}_{i \in \mathbb{N}}$, and distributed according to the probability measure $\mu$ on $\mathbb{Z}$ defined by
\begin{eqnarray}\label{def_of_mu}
\mu(\{k\})=\left \{\begin{array}{ll}
                              2d(2d-1)^{k-1}\left(\frac{d-1}{d}\right)  & \mbox{if $k=0,-1,-2,\ldots\,,$}\\
                              0                                                                & \mbox{otherwise.}
                         \end{array}
                    \right.
\end{eqnarray}
By Proposition~3.8 of \cite{resnick:1987},
\begin{equation}\label{M*}
M=\sum_i \delta_{(j_i,v_i,u_i,s_i,\mathbf{r}_i)} \sim \mbox{PRM}(v_\alpha \otimes \nu \otimes \zeta \otimes \mu \otimes \gamma)
\end{equation}
on $\big([-\infty,\infty]\setminus\{0\}\big)\times W\times G \times \mathbb{Z} \times \Gamma$.

\subsection{The weak convergence results} Let $\mathcal{M}$ be the space of all Radon measures on $[-\infty,\infty]\setminus\{0\}$ equipped with the vague topology. Since $|E_n|=\Theta((2d-1)^n)$, one expects $(2d-1)^{-n/\alpha}$ to be the correct scaling in this case. As we shall see, the partial maxima sequence \eqref{def_M_n_general} grows in this rate as well. Define the function $f'\in \mathcal{L}^\alpha(W\times G,\nu\otimes \zeta)$ based on $f$, as
$$f'(v,t)=f(v, t^{-1}) \mbox{\ \ \ for all\ \ \ } v\in W, t \in G.$$
Using \eqref{defn:thin:g}, define for each $\l \in \Z$ and for each $\xi \in \Gamma_\l$, the function $\tilde{f'}^{(\l,\xi)}$ on $W \times G$ by
$$\tilde{f'}^{(\l,\xi)}(w,t)=f'(w,t)\mathds{1}_{\{t\in \xi\}} = f(w, t^{-1})\mathds{1}_{\{t\in \xi\}}.$$
With these notations and machineries, we can now state the main theorem of this section. See Section~\ref{DP} for the proofs of all the results stated in this section.

\begin{thm}\label{maintheorem}
Let $\{X_t\}_{t \in G}$ be the mixed moving average given in \eqref{ourseries} , and define the sequence of point processes
\begin{equation}
N_n:=\sum_{t \in E_n} \delta_{(2d-1)^{-n/\alpha}X_t},\quad n=1,2,\ldots\,. \label{defn:N_n:gen}
\end{equation}
Then $N_n \Rightarrow N_*$ (as $n \rightarrow \infty$) weakly in the space $\mathcal{M}$, where $N_*$ is a randomly thinned cluster Poisson random measure with representation
\begin{equation}{\label{Nstar2}}
  N_*  =
      \sum_{i=1}^\infty \sum_{k \in G}\delta_{j_i\tilde{f'}^{\left(|u_i|, r_{i,|u_i|}\right)}(v_i,k)}\mathbf{1}_{(u_i \neq e)} + \sum_{i=1}^\infty \sum_{k \in G}\delta_{\left(\frac{d}{d-1}\right)^{1/\alpha} j_i\tilde{f'}^{\left(s_i, r_{i,s_i}\right)}(v_i,k)}\mathbf{1}_{(u_i =e)}.
\end{equation}
Here $j_i,v_i,u_i,s_i,r_i$ are as in \eqref{M*}. Furthermore $N_*$ is Radon on $[-\infty,\infty]\setminus\{0\} $ with Laplace functional
\begin{eqnarray}\label{lap2}
\E \big(e^{-N_*(g)}\big) = \exp\Bigg\{-\iint\left(\sum_{\l \in \mathbb{Z}}\frac{2d}{(2d-1)^{1-\l}}\int\left(1-e^{-\sum_{k\in G}g(x\tilde{f'}^{(\l,\xi)}(v,k))}\right)\gamma_\l(d\xi)\right)\nonumber\\
\nu_\alpha(dx)\nu(dv)\Bigg\},
\end{eqnarray}
for any nonnegative measurable function $g$ defined on $[-\infty,\infty]\setminus\{0\}$.
\end{thm}

As mentioned earlier, in case of $G=\mathbb{Z}^d$, the thinning of the function $f$ is absent due to amenability of the group. Note that in the above limit, index $(\l, \xi) \in \mathbb{Z} \times \Gamma_\l$ of the thinned function $\tilde{f'}$ becomes random. That is why we have come up with the term  \emph{randomly thinned cluster Poissson process} for the limiting point process $N_*$. We can use the convergence of the point process to get the weak convergence of partial maxima $M_n$ scaled by$ (2d-1)^{n/\alpha}$. The limit is a positive constant times the standard $\alpha$-Fr\'{e}chet distribution and the constant is, not surprisingly, much more sophisticated and involved compared to the corresponding one in case of $\mathbb{Z}^d$ obtained in \cite{samorodnitsky:2004a} and \cite{roy:samorodnitsky:2008}.

\begin{cor}\label{maxima}
Let $M_n$ be as in \eqref{def_M_n_general}. Then
$$\frac{1}{(2d-1)^{n/\alpha}}M_n \Rightarrow \mathfrak{C}_\alpha^{1/\alpha}K_X Z_\alpha,$$
where $Z_\alpha$ is a standard $\alpha$-Fr\'{e}chet random variable, $\mathfrak{C}_\alpha$ is the stable tail constant given in ~\eqref{sttail}, and
$$K_X=\left(\sum_{\l \in \mathbb{Z}} (2d)(2d-1)^{\l-1}\int_W \int_{\Gamma_\l} {2\left(\sup_{k\in G} |\tilde{f'}^{(\l,\xi)}(v,k)|\right)^\alpha}\gamma_\l(d\xi)\nu(dv)\right)^{1/\alpha} \in (0, \infty).$$
\end{cor}

\subsection{A special case with level symmetry} The above theorem takes a particularly simple form if we assume a level symmetry assumption on the function $f$, i.e., if for each $v \in W, t\in G$,
\begin{equation}\label{assumption on f}
f(v,t)=q(v,|t|),
\end{equation}
for some function $q$ on $W \times \mathbb{N}$. For each $\l \in \mathbb{Z}$, fix $\xi_\l \in \Gamma_{\l}$. Observe that by level symmetry, the thinned functions $\tilde{f'}^{(\l,\xi_{\l})}$ and $\tilde{f}^{(\l,\xi_{\l})}$ are equal. We abuse the notation slightly and denote both of these functions by $\tilde{f}^{(\l)}$. To clarify the intriguing but ultimately complex structure of the limiting point process obtained in this section, we present pictures of the $\l$-subgraphs (see Figure~\ref{fig:Subgraph}) corresponding to $\tilde{f}^{(\l)}$ for $\l=0,-1,1$ when $G=\Z*\Z$ is a free group of rank $d=2$ and $f$ satisfies the level symmetry assumption \eqref{assumption on f}. These pictures and the corollary below illustrate what random thinning means in this special case.

\begin{figure}[h]
    \centering
    \includegraphics[width=0.85\textwidth]{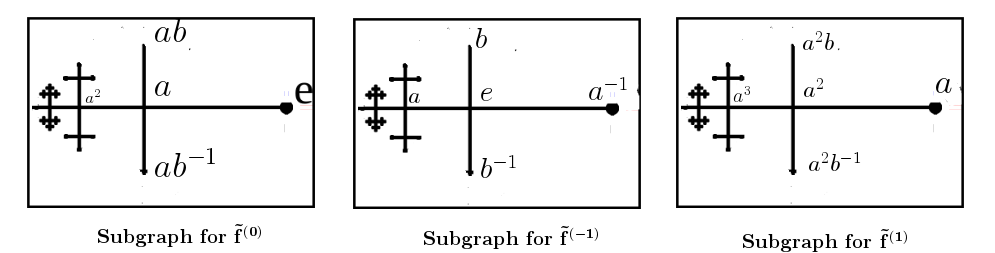}
    \caption{The $\l$-subgraphs corresponding to $\tilde{f}^{(\l)}$ for $\l=0,-1,1$}
    \label{fig:Subgraph}
\end{figure}

\begin{cor}\label{simplethm}
Let $\{X_t\}_{t \in G}$ be the mixed moving average given in \eqref{ourseries}, where f satisfies the assumption given in ~\eqref{assumption on f}. Define the sequence of point processes $N_n$ by \eqref{defn:N_n:gen}. Then
\begin{equation}{\label{Nstar}}
  N_n \Rightarrow N_* := \sum_{i=1}^\infty \sum_{t \in G}\delta_{j_i\tilde{f}^{(|u_i|)}(v_i,t)}  \mathbf{1}_{(u_i \neq e)}  + \sum_{i=1}^\infty \sum_{t \in G}\delta_{\left(\frac{d}{d-1}\right)^{1/\alpha} j_i\tilde{f}^{(s_i)}(v_i,t)}\mathbf{1}_{(u_i =e)},
\end{equation}
where $j_i,v_i,u_i$ are as in \eqref{defofPRM}, and $\{s_i\}$ are distributed independent of $(j_i,v_i,u_i)$ according to the probability measure $\mu$ as defined in \eqref{def_of_mu}. $N_*$ is Radon on $[-\infty,\infty]\setminus\{0\} $ with Laplace functional
\begin{equation}\label{lap}
\E \big(e^{-N_*(g)}\big) =\exp\left\{-\iint\sum_{\l=-\infty}^{\infty}2d(2d-1)^{\l-1}\left(1-e^{-\sum_{t\in G}g(x\tilde{f}^{(\l)}(v,t))}\right)\nu_\alpha(dx)\nu(dv)\right\}.
\end{equation}
\end{cor}

Note that if we assume $f$ satisfies the level symmetry assumtion ~\eqref{assumption on f}, then it is easy to check that the Laplace functional given in ~\eqref{lap2} reduces to the one in ~\eqref{lap}. As observed earlier, $f'=f$, and the expression $\sum_{k \in G} g(x\tilde{f}^{(\l,\xi)}(v,k))$ in the exponent of the Laplace functional in \eqref{lap2} is the same for all $\xi \in \Gamma_\l$. So the inner integral in \eqref{lap2} does not depend on the subgraph $\xi$. Since $\gamma_\l$ is a probability measure, the rest follows.

Using $f \in \mathcal{L}^\alpha(W \times G, \nu \otimes \zeta)$ and \eqref{assumption on f}, for $\nu$-almost all $v \in W$, we define functions $h_v \in \mathcal{L}^\alpha(G,\zeta)$ as follows. If $\sup_{t \in G}|f(v,t)|$ is attained at $C_k$ for some $k$, then assign $h_v|_{C_\l}=f|_{\{v\} \times C_k}$ for all $0\leq \l\leq k$. Next, if $\sup_{t \notin E_k}|f(v,t)|$ is attained at $C_{k'}$ for some $k'>k$, then define $h_v|_{C_\l}=f|_{\{v\} \times C_{k'}}$ for all $(k+1)\leq \l\leq k'$, and so on. The constant $K_X$ in Corollary ~\ref{maxima} takes the following simple form under the assumption \eqref{assumption on f}:
\[
K_X^\alpha=\frac{2^\alpha}{d-1} \int_W L(v)^{\alpha}\, \nu(dv) + \int_W\|2 h_v\|_\alpha^\alpha\,\nu(dv),
\]
where $L(v):=\sup_{t\in G}|f(v,t)|$ and $\|g\|_\alpha=(\sum_{t\in G} |g(t)|^\alpha)^{1/\alpha}$ for any function $g\in \mathcal{L}^\alpha(G,\zeta)$. Note that the first term of $K_X^\alpha$ was present (up to a constant multiple) in case of $\mathbb{Z}^d$ (see \cite{samorodnitsky:2004a} and \cite{roy:samorodnitsky:2008}) but the second term is new and can be interpreted as the contribution of non-amenability (of the group $G$) to the clustering of the extremes of $\{X_t\}_{t \in G}$.

\subsection{Open problems} We would like to mention that the results in this paper give rise to a bunch of open problems some of which will perhaps be taken up as future directions by the authors. For instance, Gennady Samorodnitsky asked the following question in a personal communication with the second author: \emph{is it possible to characterize all finitely generated countable groups for which the degenerate-nondegenerate decomposition is different from the Hopf decomposition?} While we believe that this is perhaps a difficult question, it does open a Pandora's box full of open and interesting problems. For example, it may still be possible to partially answer this question by considering special cases and eventually giving various sufficient conditions on the group so that a new transition boundary is obtained in Theorem~\ref{decomp}.

Most of the works mentioned in the second paragraph of Section~\ref{sec:intro} have not been extended to the case of random fields generated by free groups. These can also lead to many intriguing open problems relating ergodic theory (of nonsingular actions of free groups) with probability theory (of tree-indexed random fields). The non-amenability of free groups would surely affect various stochastic properties of such fields as well and it would be fascinating to analyze them. In particular, construction and investigations of max-stable random fields indexed by trees will surely turn out to be important in spatial extremes.

Since nonsingular (also called quasi-invariant) actions arise naturally in the study of Lie groups, one can think of going beyond countable groups (and $\mathbb{R}^d$), and ask similar questions for stationary stable and max-stable random fields indexed by Lie groups. Using the structure theorem of abelian groups, \cite{roy:samorodnitsky:2008} gave finer asymptotics for the partial maxima of stable random fields indexed by $\mathbb{Z}^d$ (see also \cite{chakrabarty:roy:2013} for the continuous parameter case). However such finer results are still missing in our setup mainly due to unavailability of a general structure theorem for finitely generated noncommutative groups. It is perhaps possible to resolve this issue in special classes of groups.

\section{Proofs of the results stated in Section~\ref{PM}} \label{CP}

\begin{proof}[Proof of Theorem \ref{decomp}] Let us define for any $f \in \mathcal{L}^\alpha(S,m)$,
\begin{equation*}\label{psi}
\psi(f):=\limsup_{n \rightarrow \infty} \frac{(b_n(f))^\alpha}{|E_n|},
\end{equation*}
where $b_n(f)$ is as defined in \eqref{definitionofb_n}. For any measurable $B\subseteq S$ with $m(B)<\infty$, let $\psi(B):=\psi(\mathbf{1}_B)$. Note that for any $A, B \subseteq S$ with $m(A \cup B)<\infty$,
\begin{equation}\label{propertyforsum}
\psi(A \cup B)\leq \psi(A)+\psi(B),
 \end{equation}
  and if $0 \leq f \leq g \in \mathcal{L}^\alpha(S,m)$ then
 \begin{equation}\label{monotonicity}
\psi(f)\leq \psi(g).
 \end{equation}
We will need the following lemmas.
\begin{lemma}\label{approx} If $f,g\in \mathcal{L}^\alpha$ are such that
$$\int|f(x)-g(x)|^\alpha m(dx)<\epsilon,$$
then \begin{eqnarray*}
\psi(f)&\leq & \epsilon +\psi(g)   \mbox{\ \ \ \ \ \ \ \ \ \ \ \ \ \ \  for\ \ } \alpha \in (0,1)\\
\psi^{1/\alpha}(f)&\leq & \epsilon^{1/\alpha} +\psi^{1/\alpha}(g) \mbox{\ \ \ \ for\ \ } \alpha \in [1,2)
\end{eqnarray*}
\end{lemma}
\begin{proof}
Note that
\begin{eqnarray*}
(b_n(f-g))^\alpha \leq \sum_{t\in E_n} \int |(f_t-g_t)(x)|^\alpha m (dx)
=|E_n|\int |f(x)-g(x)|^\alpha m(dx) <|E_n|\epsilon,
\end{eqnarray*}
from which this result follows using the triangle inequality, and the facts that $(x+y)^\alpha \leq x^\alpha +y^\alpha$ for all $\alpha \in (0,1)$ and $x, y \geq 0$, and that for $\alpha \in [1,2)$, $\mathcal{L}^\alpha$ is a normed space.
\end{proof}

The above lemma has the following important consequences.

\begin{cor}\label{infinitesum}
If $B\subseteq S$ with $m(B)<\infty$ can be decomposed as $B=\cup_{n=1}^\infty B_n$, where $B_i$'s are pairwise disjoint satisfying $\psi(B_n)=0$ for all $n=1,2,\ldots$, then $\psi(B)=0$.
\end{cor}
\begin{proof}
As $m(B)=\sum_{n=1}^\infty m(B_n)<\infty$, hence, given any $\epsilon>0$, we can get a sufficiently large $N \in \N$ such that $m(B\setminus \cup_{n=1}^N B_n)<\epsilon$. Applying Lemma~\ref{approx}, for $\alpha \in (0,1]$, we get
$$\psi(B) \leq \epsilon +\psi(\cup_{n=1}^N B_n) \leq \epsilon +\sum_{n=1}^N \psi(B_n)$$ by \eqref{propertyforsum}. As $\psi(B_n)=0$ for all $n$, $\psi(B)\leq \epsilon$. Since this holds for all $\epsilon >0$, we are done.
\end{proof}

\begin{cor}\label{indicatormatters} If $\psi(B)>0$ for all subsets $B$ with $0<m(B)<\infty$, then $\psi(f)>0$ for all nonzero $f\in \LLL^\alpha $.
Also if $\psi(B)=0$ for all subsets $B$ with $m(B)<\infty$, then $\psi(f)=0$ for all $f\in \LLL^\alpha$.
\end{cor}
\begin{proof}
For any nonzero $f\in \LLL^\alpha$, there exists $c>0$ and some set $C$ with $0<m(C)<\infty$ such that $|f| \geq c\ind_C$. Hence if $\psi(B)>0$ for all $B$ with $0 < m(B) <\infty$, then $\psi(C)>0$. Thus, using \eqref{monotonicity},
$$\psi(f) \geq c^\alpha\psi(C)>0.$$

Next assume $\psi(B)=0$ for all subsets $B$ with $m(B) <\infty$ and $f\in \LLL^\alpha$. Then given any $\epsilon >0$, we get $K$ large enough, and $c$ small enough, such that $\int|f(x)-f(x)\ind_{\{c\leq |f|\leq K\}}(x)|^\alpha m(dx)<\epsilon$. Also observe that $\psi(f\ind_{\{c\leq|f|\leq K\}}) \leq K^\alpha \psi (\{c\leq|f|\leq K\})=0$ as $m(c\leq|f|\leq K)\leq m(|f|\geq c)<\infty$. The result now follows by applying Lemma~\ref{approx}. \end{proof}

This lemma tells us that it is enough to compute $\psi(B)$ for all sets $B$ with $m(B)<\infty$ instead of all functions in $\LLL^\alpha$. The next lemma relates $\psi(A)$ with $\psi(\phi_t(A))$.

\begin{lemma}\label{psiinvariant} If $\psi(A)=0$ for some subset $A$ with $0<m(A)<\infty$ and $g\in G$ is such that $m(\phi_g(A))<\infty$, then $\psi(\phi_g(A))=0$.
\end{lemma}
\begin{proof}
First assume $\{\phi_t\}$ is measure $m$ preserving. Then
\begin{equation*}
(b_n(\phi_g(\ind_A)))^\alpha = m\left(\bigcup_{t\in E_n}\phi_t(\phi_g(A))\right) =m\left(\bigcup_{t\in E_n}\phi_{g.t}(A)\right)
 \leq  m\left(\bigcup_{t\in E_{n+|g|}}\phi_t(A)\right).
\end{equation*}
The last inequality follows as $|g\cdot t|\leq |g|+|t|\leq n+|t|$ and hence $\{g\cdot t|t\in E_n\}\subseteq E_{n+|g|}$.
Thus
\begin{eqnarray*}
\psi(\phi_g(A)) = \limsup_{n\rightarrow \infty} \frac{(b_n(\phi_g(\ind_A)))^\alpha}{|E_n|} &\leq & \limsup_{n\rightarrow \infty} \frac{m\left(\bigcup_{t\in E_{n+|g|}}\phi_t(A)\right)}{|E_n|}\\
&\leq & \limsup_{n\rightarrow \infty} \frac{m\left(\bigcup_{t\in E_{n+|g|}}\phi_t(A)\right)}{|E_{n+|g|}|} |E_{|g|}|\\
&=& |E_{|g|}|\psi(A)=0.
\end{eqnarray*}
Here we have used the following combinatorial fact from geometric group theory: for any finitely generated group $G$, $|E_{m+n}|\leq |E_m||E_n|$ for all $m, n \in \mathbb{N}$; see Chapter~6 of \cite{Harpe:2000}.

Now we assume $\phi_t$ is any nonsingular map (not necessarily measure preserving). We have $0<m(A)<\infty$, $m(\phi_g(A))<\infty$ and $\psi(A)=0$. Define
\begin{equation}\label{wts}
w_t(s):=\frac{d\  m\circ\phi_t}{d\ m}(s), \quad t \in G, s \in S,
\end{equation}
and the group action $\phi_t^*$ of $G$ on $(S\times (0,\infty),  m\otimes \mbox{Leb})$ as
$$\phi_t^*(s,y):=\left(\phi_t(s),\frac{y}{w_t(s)}\right), \quad s \in S, \, y>0, \, t\in G.$$

It is easy to see that $\phi_t^*$ preserves the measure $m\otimes \mbox{Leb}$ (this action is called Maharram extension; see \cite{maharam:1964} and Chapter 3.4 of \cite{aaronson:1997}). Denote for any set $B\subseteq S\times (0,\infty)$, $\psi^*(B)$ as before but using the group action $\phi_t^*$. Also note that, for any $n\in \{0,1,2,\ldots\}$, and any subset  $B\subseteq S $, $\psi(B)=\psi^*(B\times (n,n+1])$. This is because
\begin{eqnarray*}
&&m\otimes \mbox{Leb}\left(\bigcup_{t\in E_n}\phi_t^*(B\times(n,n+1])\right)\\
&=& \int_S \int_0^\infty \max_{t\in E_n} \ind(\phi_t(s)\in B,nw_t(s)< y\leq(n+1)w_t(s))dy\,m(ds)\\
&=& \int_S \max_{t\in E_n}w_t(s)\ind _B(\phi_t(s))m(ds)=(b_n(\ind_B))^\alpha.
\end{eqnarray*}
Hence $\psi^*(A\times (n,n+1])=\psi(A)=0$ for all $n=0,1,2,\ldots$. Also as $\psi^*(\phi_g(A)\times (0,1])=\psi(\phi_g(A))$, so we need to prove $\psi^*(\phi_g(A)\times (0,1])=0$. To this end, let us define $\Omega_n=\phi_g^*(A\times (n,n+1])$ and $B=\phi_g(A)\times(0,1]$. Then $\cup_{n=0}^\infty\Omega_n$ is equal to
$$\bigcup_{n=0}^\infty \phi_g^*(A\times (n,n+1])=\phi_g^*\left(\bigcup_{n=0}^\infty(A\times (n,n+1])\right)=\phi_g^*(A\times (0,\infty))=\phi_g(A)\times(0,\infty).$$

As $B\subseteq \cup_{n=0}^\infty \Omega_n$, $\Omega_n$'s are disjoint, so $B$ can be decomposed as $B=\cup_{n=0}^\infty (B\cap \Omega_n)$. Also $m\otimes \mbox{Leb}(B)=m(\phi_g(A))<\infty$, hence by Corollary~\ref{infinitesum} and \eqref{monotonicity}, it is enough to show $\psi^*(\Omega_n)=0$ for all $n$. Now
$$\psi^*(\Omega_n)=\psi^*(\phi_g^*(A\times (n,n+1]))=0$$
using the already considered case of measure preserving actions and $\psi^*(A\times(n,n+1])=\psi(A)=0$.
\end{proof}

We are now in a position to present the proof of Theorem~\ref{decomp}.

\noindent (i) For simplicity, assume without loss of generality, that the control measure $m$ is a probability measure. This can always be done because if $\nu$ is a probability measure equivalent to $m$, define $h=f(\frac{dm}{d\nu})^{1/\alpha}\in \LLL^\alpha(S,\nu)$, and write $X_t$ as an integral representation in \eqref{integrep} and \eqref{eq1} replacing $f$ by $h$ and the \sas random measure $M$ by an \sas random measure with control measure $\nu$. Note that the supports of $f$ and $h$ are equal. Also $b_n(f)$ calculated with respect to the measure $m$ is same as $b_n(h)$ corresponding to the measure $\nu$. Henceforth we assume $m$ is a probability measure.

Consider all subsets $B\subseteq S$ such that $\psi(B)=0$. As these sets form a hereditary collection (i.e., $C\subseteq B$ and $\psi(B)=0$ implies $\psi(C)=0$), we can take the measurable union of all such sets, and call it $\BB$. Define $\AA:=S\setminus \BB$. Note that $\psi(\BB)=0$ by Corollary~\ref{infinitesum} and  exhaustion lemma (see page~7 of Aaronson \cite{aaronson:1997}). Hence, for any $C\subseteq S$, $\psi(C)=0$ if and only if $C\subseteq \BB$. Consequently, $C\subseteq \AA$ if and only if for all subsets $B\subseteq C$ with $m(B)>0$, one has $\psi(B)>0$.

In order to see that $\BB$ is $\{\phi_t\}$-invariant, take any set $C\subseteq \BB$ and any $g\in G$. As $C\subseteq \BB$, $\psi(C)=0$ and then by Lemma~\ref{psiinvariant}, we have $\psi(\phi_g(C))=0$, which implies $\phi_g(C)\subseteq \BB$. This shows $\phi_g(\BB)\subseteq \BB$. As $\phi_g$ is invertible, we have $\phi_g(\BB)=\BB$. Hence $\BB$ is $\{\phi_t\}$-invariant. Hence also $\AA$ is $\{\phi_t\}$-invariant. Applying Corollary~\ref{indicatormatters}, we have, for all $f \in\LLL^\alpha$ supported on $\BB$,  $\psi(f)=0$. Similarly, for any nonempty set $B\subseteq \AA$, we have $\psi(B)>0$, and hence for any nonzero $f$ supported on $\AA$, it follows that $\psi(f)>0$. Hence for any $f$ whose support has nontrivial intersection with $\AA$, $\psi(f)\geq \psi(f\ind_\AA)>0$.

\noindent (ii) If $W^*$ is a wandering set with $m(W^*)>0$, then
\begin{eqnarray*}
(b_n(\ind_{W^*}))^\alpha=\int_S \max_{t\in E_n}w_t(s)\ind _{W^*}(\phi_t(s))m(ds)=\int_S \sum_{t\in E_n}w_t(s)\ind _{W^*}(\phi_t(s))m(ds),
\end{eqnarray*}
where $w_t(s)$ is as in \eqref{wts}. To see the last equality, note that the functions\\ $w_t(s)\ind _{W^*}(\phi_t(s))$ are supported on $\phi_{t^{-1}}(W^*)$, which are all pairwise disjoint as $t$ runs over $G$. Therefore, the maximum can be replaced by the sum. It is easy to observe that $\int_S w_t(s)\ind _{W^*}(\phi_t(s))m(ds)=m(W^*)$ and hence
\begin{equation}\label{bndiss}
(b_n(\ind_{W^*}))^\alpha=|E_n|m(W^*),
\end{equation}
which yields $\psi(W^*) >0$. In fact, for any $B\subseteq W^*$ with $m(B)>0$, $B$ is also a wandering set of positive measure and so $\psi(B)>0$, which means $W^*\subseteq \AA$. As the dissipative part $\DD$ is the union of all wandering sets, it follows that $\DD\subseteq \AA$. In fact, note that \eqref{bndiss} implies $\lim_{n \rightarrow \infty} \frac{b_n(\ind_{B})}{|E_n|^{1/\alpha}}>0$ for all $B\subseteq \DD$.

Next we want to show that the positive part $\PP\subseteq \BB$, i.e., for all $B\subseteq \PP$, $\psi(B)=0$. Restrict the group action to the set $\PP$, and assume that the measure $m$ is a probability measure which is preserved by the group action (go to the equivalent probability measure of $m$ that is preserved by the group action, which exists as $\PP$ is the positive part, and note that the decomposition into sets $\AA$ and $\BB$ remains unchanged for any equivalent measure). Then clearly $\psi(B)=0$ as $(b_n(\ind_B))^\alpha =m(\cup_{t\in E_n}\phi_t(B))\leq 1$ and $|E_n|\rightarrow \infty$.

\noindent (iii) Recall that $\{X_t\}_{t \in G}$ has an integral representation \eqref{integrep} of the form \eqref{eq1}. Let $\{Y_t\}_{t \in G}$ be another stationary \sas random field independent of $\{X_t\}$ such that its integral representation has a kernel $g\in \mathcal{L}^\alpha(S',m')$ and a $G$-action $\{\phi_t'\}$ on $S'$ satisfying $\psi(g)=0$ and $b_n(g)\geq c|E_n|^\theta$ for all $n \geq 1$. Here $c$ and $\theta$ are positive constants. Marginally, such a $\{Y_t\}_{t \in G}$ exists; see Example~\ref{eg1}). However, in order to construct such a random field independent of $\{X_t\}_{t\in G}$, an enlargement of the underlying probability space may be necessary.

Now assume that $f$ is supported on $\BB$, i.e., the component $X_t^\AA=0$. By Part~(i) of this Theorem, we have $\psi(f)=0$. Let $Z_t:=X_t + Y_t$, $t \in G$ defined in parallel to the process $\mathbf{Z}$ defined in page 1452 of \cite{samorodnitsky:2004a}. That is, $\{Z_t\}_{t\in G}$ has an integral representation on the (possibly artificially disjointified) union $S\cup S'$ with kernel $f\ind_S+g\ind_{S'}$, and the nonsingular action defined by
$$
\eta_t(s)=\left \{\begin{array}{ll}
                              \phi_t(s)  & \mbox{if $s \in S$},\\
                              \phi_t'(s)                                                              & \mbox{if $s\in S'$}
                         \end{array}
                    \right.
$$
for all $t\in G$. Then clearly, $\psi(f+g) = \psi(f)+\psi(g)=0$. Following the arguments used in the proof of (4.3) in \cite{samorodnitsky:2004a}, the rest follows.

If the component $X_t^\AA$ is nonzero, then the support of $f$ has nonzero intersection with $\AA$, and hence by Part~(i) of this theorem, $\psi(f)=\limsup_{n\rightarrow \infty}\frac{b_n(f)}{|E_n|^{1/\alpha}}=b>0$ for some positive constant $b$. Get a subsequence $n_k$ such that $\lim_{k\rightarrow \infty} \frac{b_{n_k}(f)}{|E_{n_k}|^{1/\alpha}}=b$. Working along this subsequence, and following the proof of (4.9) in \cite{samorodnitsky:2004a}, we get $M_{n_k}/b_{n_k} \Rightarrow \mathfrak{C}_\alpha^{1/\alpha}Z_\alpha$. Hence
$$\frac{M_{n_k}}{|E_{n_k}|^{1/\alpha}}\Rightarrow b\mathfrak{C}_\alpha^{1/\alpha}Z_\alpha.$$

Now we show that $M_n/|E_n|^{1/\alpha}$ is always tight. Take $\{Z_t\}_{t \in G}$ as in the proof of Part~(iii)~(a). Following arguments in the proof of (4.3) in \cite{samorodnitsky:2004a}, it can be shown that $M_n^Z/b_n^Z$ is tight (here $M_n^Z$, $b_n^Z$ denote the quantities corresponding to $M_n$, $b_n$ defined for the stationary \sas random field $\{Z_t\}_{t \in G}$ ). Also for any nonzero $f\in \LLL^\alpha$,
 $$b_n(f)=\left(\int \max _{t\in E_n}|f_t(x)|^\alpha m(dx)\right)^{1/\alpha} \leq \left(\sum_{t\in E_n}\int|f_t(x)|^\alpha m(dx)\right)^{1/\alpha}=|E_n|^{1/\alpha} \|f\|_\alpha.$$
This calculation yields that $b_n^Z \leq |E_n|^{1/\alpha}\big(\|f\|_\alpha^\alpha + \|g\|_\alpha^\alpha \big)^{1/\alpha}$, where $g$ is the kernel function for $\{Y_t\}_{t \in G}$. Hence $M_n^Z/|E_n|^{1/\alpha}$ is tight. This implies, as in the proof of (4.3) in \cite{samorodnitsky:2004a}, that $M_n/|E_n|^{1/\alpha}$ is also tight.
\end{proof}

\begin{proof}[Proof of Theorem \ref{freegroupliminf}]
We show that when $G$ is a free group of rank $d$, then for any $f\in \LLL^\alpha$,
\begin{equation}\label{liminfsup}
\mbox{if} \quad \liminf_{n\rightarrow \infty}\frac{b_n(f)}{|E_n|^{1/\alpha}}=0, \quad \mbox{then} \quad \limsup_{n\rightarrow \infty}\frac{b_n(f)}{|E_n|^{1/\alpha}}=0.
\end{equation}
To see this, observe that, for any $m,n \in \N$, $E_{n+m}=E_n\cdot E_m:=\{g_1\cdot g_2: g_1\in E_n,g_2\in E_m\}$. Hence
\begin{eqnarray*}
(b_{n+m}(f))^\alpha &=&
\int \max_{t_1\in E_n}\max _{t_2\in E_m}|f_{t_1t_2}(x)|^\alpha m(dx)\\
&\leq&  \sum_{t_1\in E_n}\int \max _{t_2\in E_m}|f_{t_1t_2}(x)|^\alpha m(dx)
=|E_n|(b_m(f))^\alpha.
\end{eqnarray*}
Here the last equality follows because with $w_t$ as in \eqref{wts},
\begin{eqnarray*}
&&\int \max _{t_2\in E_m}|f_{t_1t_2}(x)|^\alpha m(dx)\\
&=&\int \max_{t_2\in E_m}w_{t_1}(x)w_{t_2}(\phi_{t_1}(x))|f\circ \phi_{t_2}(\phi_{t_1}(x))|^\alpha dm(x)\\
&=&\int\max_{t_2\in E_m}w_{t_2}(x)|f\circ\phi_{t_2}(x)|^\alpha dm(x) =\int\max_{t_2\in E_m}|f_{t_2}(x)|^\alpha dm(x).
\end{eqnarray*}

Hence for any $m,n \in \N$,
\begin{equation}\label{boundonbn}
\frac{(b_{n+m}(f))^\alpha}{|E_{n+m}|}\leq \frac{|E_n||E_m|}{|E_{n+m}|}\frac{(b_m(f))^\alpha}{|E_m|}\leq \left(\frac{d}{d-1}\right)^2 \frac{(b_m(f))^\alpha}{|E_m|}
\end{equation}
using the trivial bounds
\begin{equation}\label{trivial_bounds}
(2d-1)^n\leq |E_n|\leq \frac{d}{d-1}(2d-1)^n.
\end{equation}
So if $\liminf_{n\rightarrow \infty}\frac{b_n(f)}{|E_n|^{1/\alpha}}=0$, then for any $\epsilon>0$, we get an $m\in \N$, such that $\frac{b_m(f)}{|E_m|^{1/\alpha}}<\epsilon$. Hence using \eqref{boundonbn}, we have that for all $n\geq m$, $\frac{b_n(f)}{|E_n|^{1/\alpha}}<\left(\frac{d}{d-1}\right)^{2/\alpha}\epsilon$. This shows \eqref{liminfsup}.

Thus for any $f\in\LLL^\alpha$ whose support has nontrivial intersection with $\AA$, we have $\liminf_{n\rightarrow \infty}\frac{b_n(f)}{|E_n|^{1/\alpha}}>0$. Again, as $0\leq \frac{b_n(f)}{|E_n|^{1/\alpha}}\leq \|f\|_\alpha$, given any sequence $n_k$, there exists a further subsequence $n_{k_\l}$ such that $\frac{b_{n_{k_\l}}(f)}{|E_{n_{k_\l}}|^{1/\alpha}}$ converges to some $c>0$. Rest follows by applying the proof of (4.9) in \cite{samorodnitsky:2004a} along this subsequence.
\end{proof}

\begin{proof}[Proof of Theorem \ref{shiftofboundary}] This result trivially follows from Examples \ref{eg2} and \ref{eg3}.
\end{proof}

\begin{proof}[Proof of Theorem \ref{boundarymatters}]
One implication is obvious. For the other one, assume $\limsup_{n\rightarrow \infty} \frac{\int \max_{t\in C_n}|f_t(x)|^\alpha m(dx)}{(2d-1)^{n}}=0$.  Fix $\epsilon >0$, and choose $K\in \N$ large enough so that $\frac{|E_{n-K}|}{(2d-1)^n}\leq \epsilon$ for all $n \ge K$. This is possible by \eqref{trivial_bounds}. Hence
\begin{eqnarray*}
&&\limsup_{n\rightarrow \infty} \frac{\int \max_{t\in (\cup_{i=n-K+1}^n C_n)}|f_t(x)|^\alpha m(dx)}{(2d-1)^{n}}\\
&\leq & \limsup_{n\rightarrow \infty}\sum_{i=0}^{K-1}\frac{\int \max_{t\in C_{n-i}}|f_t(x)|^\alpha m(dx)}{(2d-1)^{n}}\\
&\leq &\sum_{i=0}^{K-1}\frac{1}{(2d-1)^i}\limsup_{n\rightarrow \infty}\frac{\int \max_{t\in C_{n-i}}|f_t(x)|^\alpha m(dx)}{(2d-1)^{n-i}}=0.
\end{eqnarray*}
Also for all $n \ge K$,
\begin{equation*}
\frac{\int \max_{t\in E_{n-K}}|f_t(x)|^\alpha m(dx)}{(2d-1)^{n}}
\leq \|f\|_\alpha^\alpha \frac{|E_{n-K}|}{(2d-1)^n}\leq \epsilon \|f\|_\alpha^\alpha.
\end{equation*}
Putting the two inequalities together, we get $\limsup_{n\rightarrow \infty} \frac{\int \max_{t\in E_n}|f_t(x)|^\alpha m(dx)}{(2d-1)^{n}}\leq \epsilon \|f\|_\alpha^\alpha$. Since $\epsilon > 0$ is arbitrary, this result follows.
\end{proof}

\begin{proof}[Proof of Theorem \ref{setsinA}]
Take a set $B\subseteq S$ such that $\limsup_{n\rightarrow \infty} \frac{a_n(B)}{|E_n|}>0$. Then for any $C\subseteq B$ also, we have $\limsup_{n\rightarrow \infty} \frac{a_n(C)}{|E_n|}>0$. For any such set $C$, let $S_n$ denote a subset of $E_n$ with cardinality $a_n(C)$ such that $\phi_{t^{-1}}(C)$, $t\in S_n$ are pairwise disjoint. Hence with $w_t$ as in \eqref{wts}, we get
\begin{equation*}
(b_n(\ind_{C}))^\alpha
\geq  \int\max_{t\in S_n}w_t(s)\ind _{C}(\phi_t(s))m(ds)
= \sum_{t\in S_n} \int w_t(s)\ind _{C}(\phi_t(s))m(ds)=a_n(C)m(C).
\end{equation*}
The above calculation yields that $\psi(C)>0$ for all $C\subseteq B$ satisfying $m(C) > 0$. Therefore $B\subseteq \AA$.
\end{proof}

\section{Proofs of the results stated in Section~\ref{Disscase}}\label{DP}

\begin{proof}[Proof of Theorem ~\ref{maintheorem}]
 We first show that the Laplace functional of $N_*$ defined in \eqref{Nstar2} is indeed of the form \eqref{lap2}. By Proposition~5.4 of \cite{resnick:2007}, we can assume, without loss of generality, that the function $g$ is also continuous with compact support. For such a function $g$, we have $\E \big(e^{-N_*(g)}\big)=\E \big(e^{-\sum_i \psi (j_i,v_i,u_i,s_i,\mathbf{r_i})}\big)$, where $\psi(x,v,u,s, \mathbold{\rho})$ is given by
$$\sum_{k\in G} g\left(x\tilde{f'}^{(|u|,\rho_{|u|})}(v,k)\right)\mathbf{1}_{(|u|>0)}
+\sum_{k\in G} g\left(\left(\frac{d}{d-1}\right)^{1/\alpha} x\tilde{f'}^{(s,\rho_{s})}(v,k)\right)\mathbf{1}_{(|u|=0)}$$
for all $x \in [-\infty, \infty]\setminus \{0\}$, $v \in W$, $u \in G$, $s \in \mathbb{Z}$ and $\mathbold{\rho}=(\rho_\l: \l \in \mathbb{Z}) \in \Gamma$. Because of ~\eqref{M*}, the Laplace functional of $N_*$ is equal to $\E(e^{-M(\psi)})$ which is
$$\exp \Bigg[-\int_{\Gamma} \int_W \int_{|x|>0}\Bigg\{ \sum_{|u|>0}\sum_{s=-\infty}^0 2d(2d-1)^{s-1}\left(\frac{d-1}{d}\right)\left(1-e^{-\sum_{k\in G} g\left(x\tilde{f'}^{(|u|,\rho_{|u|})}(v,k)\right)}\right)$$
 $$+\sum_{s=-\infty}^0 2d(2d-1)^{s-1}\left(\frac{d-1}{d}\right)\left(1-e^{-\sum_{k\in G} g\left(\left(\frac{d}{d-1}\right)^{1/\alpha} x\tilde{f'}^{(s,\rho_{s})}(v,k)\right)}\right)\Bigg\}d\nu_\alpha d\nu d\gamma(\mathbold{\rho})\Bigg].$$
The first integral inside the exponent equals
\begin{eqnarray*}
&& \; \iiint\sum_{|u|>0}\left(1-e^{-\sum_{k\in G} g\left(x\tilde{f'}^{(|u|,\rho_{|u|})}(v,k)\right)}\right)d\gamma d\nu_\alpha d\nu \\
&=&\iint\sum_{\l=1}^{\infty}\int \sum_{u\in C_\l}\left(1-e^{-\sum_{k\in G} g\left(x\tilde{f'}^{(\l,\rho_{\l})}(v,k)\right)}\right)d\gamma d\nu_\alpha d\nu \\
&=&\iint\sum_{\l=1}^\infty \left[2d(2d-1)^{\l-1}\int\left(1-e^{-\sum_{k\in G} g\left(x\tilde{f'}^{(l,\xi)}(v,k)\right)}\right)d\gamma_\l(\xi)\right]d\nu_\alpha d\nu.
\end{eqnarray*}
Here the last equality follows because $|C_\l|=2d(2d-1)^{\l-1}$. Using the change of variable $y=\left(\frac{d}{d-1}\right)^{1/\alpha}x$ and the scaling property of $\nu_\alpha$ for the second integral, and combining the output with the first one, we get that the Laplace functional of $N_*$ is indeed the one given in  \eqref{lap2}.

Next we show that $N_*$ is Radon. Note that for all $\l \in \Z$ and for all $\xi \in \Gamma_\l$,
\begin{equation}\label{easybound}
|\tilde{f'}^{(\l,\xi)}|\leq |f'| \in \mathcal{L}^\alpha(W\times G,\nu\otimes \zeta).
\end{equation}
Now fix a $\delta>0$. Set $C=[-\infty,-\delta)\cup(\delta,\infty]$ and $h=\mathbf{1}_C$. To establish that $N_*$ is Radon, it is enough to show that $\E[N_*(h)]<\infty$. To this end, we write $N_*=N^{(1)}+N^{(2)}$, where
\begin{equation*}
  N^{(1)}  =
      \sum_{i=1}^\infty \sum_{k \in G}\delta_{j_i\tilde{f'}^{\left(|u_i|,r_{i,|u_i|}\right)}(v_i,k)}\mathbf{1}_{(u_i \neq e)}.
\end{equation*}
and
\begin{equation*}
  N^{(2)}  =
      \sum_{i=1}^\infty \sum_{k \in G}\delta_{\left(\frac{d}{d-1}\right)^{1/\alpha} j_i\tilde{f'}^{\left(s_i,r_{i,s_i}\right)}(v_i,k)}\mathbf{1}_{(u_i =e)}.
\end{equation*}
We first establish that $ \E[N^{(2)}(h)]$ is finite. To this end, note that another use of the scaling property of $\nu_\alpha$ yields
\begin{align*}
&\;\;\;\;\;\quad \E[N^{(2)}(h)] \\
&\;\;\;= \iint\sum_{\l=-\infty}^0\sum_{k\in G}2d(2d-1)^{\l-1}\left(\int \mathbf{1}_C\left(y\tilde{f'}^{(\l,\rho_\l)}(v,k)\right)\nu_\alpha(dy)\right)\nu(dv)\gamma(d\mathbold{\rho}),\\
\intertext{from which, using the definition of $\nu_\alpha$ and applying \eqref{easybound}, we get}
& \;\;\; = 2\delta^{-\alpha}\int\sum_{\l=-\infty}^0\sum_{k\in G}2d(2d-1)^{\l-1}\int|\tilde{f'}^{(\l,\xi)}(v,k)|^\alpha\gamma_\l(d\xi)\nu(dv) \\
& \;\;\; \leq  2\delta^{-\alpha}\left(\frac{d}{d-1}\right)\|f'(v,k)\|_\alpha^\alpha < \infty.
\end{align*}

The proof of finiteness of $\E[N^{(1)}(h)]$, however, is slightly more involved. Let us start by observing that a similar calculation as above yields
\begin{equation}\label{form_of_N^1_h}
\E[N^{(1)}(h)] =2\delta^{-\alpha}\int\sum_{\l=1}^\infty\sum_{k\in G}2d(2d-1)^{\l-1}\int|\tilde{f'}^{(\l,\xi)}(v,k)|^\alpha\gamma_\l(d\xi)\nu(dv),
\end{equation}
which needs to be tightly estimated by a quantity that can be shown to be finite. This would require the following combinatorial fact.

\begin{lemma} \label{lemma:vertex_counting}
Fix $\l\in \N$. Then for each $k=0,1,2,\ldots$\,, every $\l$-subgraph has exactly $(2d-1)^{\floor*{k/2}}$ many vertices from $C_{\l+k}$. Here $\floor*{\cdot}$ denotes the floor function.
\end{lemma}
\begin{proof} Fix any subgraph in $\Gamma_\l$. Consider a self avoiding path $\{v_0,v_1,v_2,\ldots
\}$ (diverging away from the root) corresponding to this subgraph. Note that since every vertex in $G$ has degree $2d$, the subgraph has exactly one vertex in $C_\l$, one vertex in $C_{\l+1}$, and $2d-1$ vertices in $C_{\l+2}$.

In general, if $k \ge 4$ is even, then $V_{\frac{k-2}{2}}$ does not contain any vertex in $C_{\l+k}$, but $V_{\frac{k}{2}}$ contains vertices that lie in $C_{\l+k}$ (recall $|v_s|=\l+s$, and $V_s$ consists of all vertices at distance not greater than $s$ from $v_s$). Since every vertex has degree $2d$ and the distance from the root increases if one moves from one vertex to any of its adjacent vertices away from the root, the number of vertices in $C_{\l+k}$ that are contained in $V_{\frac{k}{2}}$ (and hence in the $\l$-subgraph) is the number of vertices at a distance $\frac{k}{2}$ from $v_k$ and away from the root (i.e., along $2d-1$ directions), which is $(2d-1)^{{k}/{2}}$. Note that any vertex in $C_{\l+k}$ that is contained in $V_s$ for some $s$ greater than ${k}/{2}$, is already contained in $V_{\frac{k}{2}}$. Thus the subgraph has $(2d-1)^{{k}/{2}}$ many vertices from $C_{\l+k}$.

Similarly, if $k \ge 3$ odd, then note that $V_{\frac{k-1}{2}}$ does not contain any vertex lying in $C_{\l+k}$ while $V_{\frac{k+1}{2}}$ does. The number of vertices in $C_{\l+k}$ that are contained in $V_{\frac{k+1}{2}}$ is the number of vertices at a distance $\frac{k-1}{2}$ from $v_k$ (along $2d-1$ directions), which is $(2d-1)^{\frac{k-1}{2}}$. This completes the proof.
\end{proof}

Now we go back and show that $\E[N^{(1)}(h)] < \infty$. Note that $\tilde{f'}^{(\l,\xi)}(v,k)$ vanishes in $C_s$ for all $s \leq \l-1$. Keeping this in mind, for each fixed $v \in W$, we consider the quantity
\begin{align*}
& \quad \sum_{\l=1}^\infty\sum_{k\in G}(2d-1)^{\l-1}\int|\tilde{f'}^{(\l,\xi)}(v,k)|^\alpha\gamma_\l(d\xi)\\
& = \sum_{\l=1}^\infty \sum_{s \ge \l} \sum_{k\in C_s}(2d-1)^{\l-1}\int|\tilde{f'}^{(\l,\xi)}(v,k)|^\alpha\gamma_\l(d\xi)\\
& =\sum_{s=1}^\infty \sum_{\l=1}^s (2d-1)^{\l-1}  \int\sum_{k\in C_s}|\tilde{f'}^{(\l,\xi)}(v,k)|^\alpha\gamma_\l(d\xi)\\
& = \sum_{s=1}^\infty \sum_{\l=1}^s (2d-1)^{\l-1} \frac{1}{|\Gamma_\l^{(s)}|} \sum_{\xi \in \Gamma_\l^{(s)}}\sum_{k\in C_s}|\tilde{f'}^{(\l,\xi)}(v,k)|^\alpha,\\
\intertext{which, by vitrue of Lemma~\ref{lemma:vertex_counting} and symmetry, reduces to}
& = \sum_{s=1}^\infty \sum_{\l=1}^s (2d-1)^{\l-1} \frac{1}{|\Gamma_\l^{(s)}|} \frac{|\Gamma_\l^{(s)}|(2d-1)^{\floor*{\frac{s-\l}{2}}}}{|C_s|}\sum_{k\in C_s}|f'(v,k)|^\alpha \\
& = \frac{1}{2d}\sum_{s=1}^\infty \left(\sum_{k\in C_s}|f'(v,k)|^\alpha \sum_{\l=1}^s(2d-1)^{-\left(s-\l-\floor*{\frac{s-\l}{2}}\right)}\right)\\
&\leq \frac{1}{2d}\sum_{s=1}^\infty  \sum_{k\in C_s}|f'(v,k)|^\alpha\sum_{r=0}^{\infty}(2d-1)^{-\left(r-\floor*{{r}/{2}}\right)}\leq  C \sum_{k\in G}|f'(v,k)|^\alpha
\end{align*}
for a positive constant $C$. Combining the above calculations with \eqref{form_of_N^1_h} finiteness of $\E[N^{(1)}(h)]$ follows because $f'\in \mathcal{L}^\alpha(W\times G,\nu\otimes \zeta)$. Hence we get that $N_*$ is Radon.

Finally, we need to prove that $N_n$ converges to $N_*$ weakly.  Inspired by \cite{Resnick:Samorodnitsky:2004}, one can guess that very few of the Poisson points $j_i$ in the definition of $X_t$ are likely to be large enough, so that $N_n$ has the same weak limit as
\begin{eqnarray*}\label{N2}
N_n^{(2)}:=\sum_{i=1}^\infty \sum_{k\in E_n} \delta_{(2d-1)^{-n/\alpha}j_if(v_i,k^{-1}u_i)}\,.
\end{eqnarray*}
The Laplace functional of $N_n^{(2)}$ can be computed easily. Using the scaling property of $\nu_\alpha$ once again we get, for $g\geq 0$ continuous with compact support, $\E e^{-N_n^{(2)}(g)}$ is equal to
\begin{eqnarray}\label{lapN2}
&\ \ \ & \exp\left\{ -\iint \frac{1}{(2d-1)^n}\sum_{u\in G}\left(1-e^{-\sum_{k\in E_n}g(xf(v,k^{-1}u)}\right)\nu_\alpha(dx)\nu(dv)\right\}
\end{eqnarray}
and by Theorem~5.2 of \cite{resnick:2007}, this needs to be shown to converge to ~\eqref{lap2} as $n \to \infty$. This is done in three steps. First we prove the convergence of $N_n^{(2)}$ when $f$ is compactly supported in the second variable. Then we remove the assumption of compact support. Finally, we verify $N_n$ to have the same limit by showing that its vague distance from $N_n^{(2)}$ converges to zero in probability. Since this proof closely follows that of Theorem~3.1 in  \cite{Resnick:Samorodnitsky:2004}, only the first step is detailed here and the rest are sketched.

Keeping the above discussion in mind, we assume first that the function $f$ in \eqref{ourma} is compactly supported in the second variable, i.e., for some positive integer $m$,
$$ f(v,u)=0 \mbox{ for all $(v,u)\in W\times G$ such that $u \notin E_m$.} $$
Fix $n>2m$ and $v \in W$. We examine the integrand in ~\eqref{lapN2} for each $u$. Since for any fixed $u \in G$,
\[
\bar{B}_u:=\{u^{-1}k: k\in E_n\}
\]
is the set of all vertices that are at a distance not greater than $n$ from $u^{-1}$, whenever  $u \notin E_{n+m}$, (then $u^{-1}$ is also not in $E_{n+m}$), due to the assumption on the support of $f$, $f(v,u^{-1}k)=0$ for any $k \in E_n$, and hence the integrand vanishes. So, we only examine the integrand with $u$ restricted to $E_{n+m}$.

Now for each $u \in E_{n-m}$, it is easy to see that $\bar{B}_u \supseteq E_m$. Also as $f$ vanishes outside $E_m$, so does $f'$, and whenever $f$ is zero, $g$ is also zero. Hence, for any $u \in E_{n-m}$,
$$
\left(1-e^{-\sum_{k\in E_n}g(xf(v,k^{-1}u))}\right)=\left(1-e^{-\sum_{k\in E_n}g(xf'(v,u^{-1}k))}\right)=\left(1-e^{-\sum_{k\in G}g(xf'(v,k))}\right),$$
which immediately turns the integral inside the exponent of \eqref{lapN2} into
\begin{eqnarray*}
&&\iint\frac{1}{(2d-1)^n}\sum_{u \in G}\left(1-e^{-\sum_{k\in E_n}g(xf'(v,u^{-1}k)}\right)\nu_\alpha(dx)\nu(dv)\\
&=& \iint\left[\sum_{\l=-(m-1)}^m \frac{1}{(2d-1)^n} \sum_{u \in C_{n+\l}}\left(1-e^{-\sum_{k\in E_n}g(xf'(v,u^{-1}k)}\right) \right. \\
&&\;\;\; + \left. \frac{1}{(2d-1)^n}|E_{n-m}|\left(1-e^{-\sum_{k\in G}g(xf'(v,k))}\right) \right]\nu_\alpha(dx)\nu(dv).
\end{eqnarray*}

Fix $\l \in \{-(m-1),-(m-2),\ldots ,0,1,\ldots,m\}$. Then fix a vertex $u \in C_{n+\l}$, and consider the subgraph $\bar{B}_u$. The collection of subgraphs as $u$ ranges over $C_{n+\l}$ is precisely the set $\Gamma_\l^{(m)}$, the collection of $m$-essentially distinct subgraphs in $\Gamma_\l$ as defined earlier. As there are $|\Gamma_\l^{(m)}|$ many $m$-essentially distinct subgraphs, by symmetry, each subgraph $\xi \in \Gamma_\l^{(m)}$ is repeated $|C_{n+\l}|/|\Gamma_\l^{(m)}|$ many times as $u$ runs over $C_{n+\l}$. Also if two $\l$-subgraphs $\xi$ differ only outside of $E_m$, then as $f'$ vanishes outside $E_m$, $\sum_{k\in G}g(x\tilde{f'}^{(\l,\xi)}(v,k))$ is same for them. Hence
\begin{eqnarray*}
&&\frac{1}{(2d-1)^n}\sum_{u \in C_{n+\l}}\left(1-e^{-\sum_{k\in E_n}g(xf'(v,u^{-1}k)}\right)\\
&=&\frac{1}{(2d-1)^n}\frac{|C_{n+\l}|}{|\Gamma_\l^{(m)}|}\sum_{\xi \in \Gamma_\l^{(m)}} \left(1-e^{-\sum_{k\in G}g(x\tilde{f'}^{(\l,\xi)}(v,k)}\right)\\
&=& 2d(2d-1)^{\l-1}\frac{1}{|\Gamma_\l^{(m)}|}\sum_{\xi \in \Gamma_\l^{(m)}} \left(1-e^{-\sum_{k\in G}g(x\tilde{f'}^{(\l,\xi)}(v,k)}\right)\\
&=& 2d(2d-1)^{\l-1}\int_{\Gamma_\l}\left(1-e^{-\sum_{k\in G}g(x\tilde{f'}^{(\l,\xi)}(v,k)}\right)d\gamma_\l(\xi).
\end{eqnarray*}
Here the last equality holds because $\gamma_\l$ restricted to $E_m$ is uniform on all possible subgraphs in $\Gamma_\l^{(m)}$.

Combining the above calculations with \eqref{easybound}, \eqref{trivial_bounds}, and using the fact that $g$ is continuous with compact support, it can be shown that \eqref{lapN2} converges to \eqref{lap2} as we let $n \to \infty$ and then $m \to \infty$. The justification of this truncation can be given using a convergence together argument in parallel to the one given in the proof of Theorem~3.1 in \cite{Resnick:Samorodnitsky:2004}. Finally, we verify that the vague metric between $N_n^{(2)}$ and $N_n$ converges to zero in probability following verbatim the corresponding portion in the aforementioned reference. This finishes the proof.
\end{proof}

\begin{proof}[Proof of Corollary~\ref{maxima}] Note that for each $s>0$, $[-\infty,-s)\cup(s,\infty]$ is a relatively compact set in $[-\infty,\infty]\setminus\{0\}$ and $N_*(\{s,-s,\infty,-\infty\})=0$ almost surely. Therefore using Theorems~3.1 and 3.2 of \cite{resnick:2007} and $N_n \Rightarrow N_*$, we have $$N_n([-\infty,-s)\cup(s,\infty]) \Rightarrow N_*([-\infty,-s)\cup(s,\infty])
\quad \mbox{ for any } s>0$$
In particular, for any $s>0$,
\begin{eqnarray*}
&&\P\left(\frac{1}{(2d-1)^{n/\alpha}}M_n \leq s\right)\\
&=&
\P\left(N_n([-\infty,-s)\cup(s,\infty])=0\right) \rightarrow \P(N_*([-\infty,-s)\cup(s,\infty])=0).
\end{eqnarray*}

To find $\P(N_*([-\infty,-s)\cup(s,\infty])=0)$, we use the Laplace functional of $N_*$. Note that as $N_*$ is a point process and hence $N_*([-\infty,-s)\cup(s,\infty])$ is a nonnegative integer valued random variable, by dominated convergence theorem,
\begin{eqnarray*}
&&\P(N_*([-\infty,-s)\cup(s,\infty])=0)\\
&=& \lim_{t\rightarrow \infty}\E \big(e^{-tN_*([-\infty,-s)\cup(s,\infty])}\big)\\
&=&\exp\left\{-\iint\sum_\l 2d(2d-1)^{\l-1}\int\lim_{t\rightarrow \infty}\left(1-e^{-\sum_{k}g(x\tilde{f'}^{(\l,\xi)}(v,k))}\right)d\gamma_\l d\nu_\alpha d\nu\right\},
\end{eqnarray*}
where $g=t\mathbf{1}_{[-\infty,-s)\cup(s,\infty]}$. Again,
$$\lim_{t\rightarrow \infty}\left(1-e^{-t\sum_{k\in G}\mathbf{1}_{[-\infty,-s)\cup(s,\infty]}(x\tilde{f'}^{(\l,\xi)}(v,k))}\right)= \mathbf{1}_{\left(|x|\geq {s}/{\sup_{k\in G} |\tilde{f'}^{(\l,\xi)}(v,k))|}\right)} ,$$
and
 $$\iint\mathbf{1}_{\left(|x|\geq {s}/{\sup_{k\in G} |\tilde{f'}^{(\l,\xi)}(v,k))|}\right)}\nu_\alpha(dx)\nu(dv)= \int\frac{2\left(\sup_{k\in G} |\tilde{f'}^{(l,\xi)}(v,k)|\right)^\alpha}{s^\alpha}d\nu.$$
Hence $\P(N_*([-\infty,-s)\cup(s,\infty])=0)$ is given by
\begin{eqnarray*}
&&\exp\left\{-\frac{\sum_{l=-\infty}^{\infty} (2d)(2d-1)^{l-1}\int_W \int {2\left(\sup_{k} |\tilde{f'}^{(l,\xi)}(v,k)|\right)^\alpha}\gamma_\l(d\xi)\nu(dv)}{s^\alpha} \right\}\\
&=& \exp\left\{-\frac{K_X^\alpha}{s^\alpha}\right\}.
\end{eqnarray*}
Finiteness of $K_X$ can be established following the argument that was used to prove $N_*$ is Radon.
\end{proof}

\begin{proof}[Proof of Corollary~\ref{simplethm}] It has already been verified that the Laplace functional in \eqref{lap2} reduces to that in \eqref{lap} under the assumption of level symmetry. We just need to show that \eqref{lap} is indeed the Laplace functional of $N_*$ defined in ~\eqref{Nstar}. This can be done exactly as in the first part of the proof of Theorem~\ref{maintheorem}. The details are skipped.
\end{proof}

\section*{Acknowledgements} A significant portion of this work was carried out in the master's dissertation of the first author at Indian Statistical Institute. He would like to thank the institute and the professors there. The authors are grateful to Gennady Samorodnitsky for a number of useful discussions, to Kingshook Biswas and Mahan Mj for drawing their attention to the group action considered in Example~\ref{eg2}, and to Mrinal Kanti Das for suggesting the reference~\cite{aluffi:2009}. The authors would also like to thank the anonymous referees for their careful reading and detailed comments which significantly improved the paper.

\end{document}